\documentclass[11pt]{amsart}

\usepackage{amssymb}
\usepackage{amsmath}
\usepackage{amsthm}
\usepackage{color} 
\usepackage{verbatim}

\newtheorem{theorem}{Theorem}[section]
\newtheorem{lemma}{Lemma}[section]
\newtheorem{remark}{Remark}[section]
\newtheorem{assumption}{Assumption}[section]

\newtheorem{definition}{Definition}[section]

\numberwithin{equation}{section}

\newcommand{\Rd}{\mathbb{R}^d}
\newcommand{\ve}{\varepsilon}

\newcommand{\sxs}{\sigma_{(\xi_s)}}
\newcommand{\bxs}{b_{(\xi_s)}}

\newcommand{\szs}{\sigma_{(\eta_s)}}
\newcommand{\bzs}{b_{(\eta_s)}}

\newcommand{\sxxs}{\sigma_{(\xi_s)(\xi_s)}}
\newcommand{\bxxs}{b_{(\xi_s)(\xi_s)}}

\newcommand{\sxt}{\sigma_{(\xi_t)}}

\newcommand{\uxt}{u_{(\xi_t)}}

\newcommand{\uzt}{u_{(\eta_t)}}

\newcommand{\uxxt}{u_{(\xi_t)(\xi_t)}}

\newcommand{\px}{\psi_{(\xi)}}

\newcommand{\pxt}{\psi_{(\xi_t)}}

\newcommand{\psk}{\psi_{(\sigma^k)}}
\newcommand{\mypsi}{\psi_{(\sigma^i)}}

\newcommand{\pxsop}{\frac{\psi_{(\xi)}^2}{\psi}}
\newcommand{\pxtsop}{\frac{\psi_{(\xi_t)}^2}{\psi}}
\newcommand{\pxsops}{\frac{\psi_{(\xi)}^2}{\psi^2}}
\newcommand{\pxtsops}{\frac{\psi_{(\xi_t)}^2}{\psi^2}}

\newcommand{\uxzxz}{|u_{(\xi_0)}(x_0)|}
\newcommand{\uxzxzxz}{|u_{(\xi_0)(\xi_0)}(x_0)|}
\newcommand{\bwxzxz}{\mathrm{B}_1(x_0,\xi_0)}
\newcommand{\btxzxz}{\mathrm{B}_2(x_0,\xi_0)}
\newcommand{\sbwxzxz}{\sqrt{\mathrm{B}_1(x_0,\xi_0)}}
\newcommand{\sbtxzxz}{\sqrt{\mathrm{B}_2(x_0,\xi_0)}}
\newcommand{\uxxobw}{\frac{|u_{(\xi)(\xi)}(x)|}{\mathrm{B}_1(x,\xi)}}
\newcommand{\uxxobt}{\frac{|u_{(\xi)(\xi)}(x)|}{\mathrm{B}_2(x,\xi)}}
\newcommand{\uxosbw}{\frac{|u_{(\xi)}(x)|}{\sqrt{\mathrm{B}_1(x,\xi)}}}
\newcommand{\uxosbt}{\frac{|u_{(\xi)}(x)|}{\sqrt{\mathrm{B}_2(x,\xi)}}}

\begin{document}
\title[Quasiderivative method for derivative estimates]{Quasiderivative method for derivative estimates of solutions to degenerate elliptic equations}
\author{Wei Zhou}
\address{127 Vincent Hall, 206 Church St. SE, Minneapolis, MN 55455}

\maketitle

\begin{abstract}
\noindent
We give an example of quasiderivatives constructed by random time change, Girsanov's Theorem and Levy's Theorem. As an application,  we investigate the smoothness and estimate the derivatives up to second order for the probabilistic solution to the Dirichlet problem for the linear degenerate elliptic partial differential equation of second order, under the assumption of non-degeneracy with respect to the normal to the boundary and an interior condition to control the moments of quasiderivatives, which is weaker than non-degeneracy. 
\end{abstract}

%%%%%%%%%%%%%%%%section1
\section{Introduction and Background}

%Partial differential equations of elliptic and parabolic type and stochastic processes of diffusion type are intimately connected. The solutions of many elliptic and parabolic partial differential equations have their  probabilistic representations, which makes it possible to apply probabilistic arguments to study these partial differential equations.

We consider the Dirichlet problem for the linear degenerate elliptic partial differential equation of second order
\begin{equation}
\left\{\begin{array}{rcll}
Lu(x)-c(x)u(x)+f(x)&=&0  &\text{in }D\\ 
u&=&g  &\text{on }\partial D,
\end{array}
\right.
\label{1a}
\end{equation}
where $Lu(x):=a^{ij}(x)u_{x^ix^j}(x)+ b^i(x)u_{x^i}(x)$, with $a=(1/2)\sigma\sigma^*$, and summation convention is understood. The probabilistic solution of (\ref{1a}) is known as
\begin{equation}
u(x)=E\bigg[g\big(x_{\tau}(x)\big)e^{-\phi_\tau}+\int_0^{\tau}f\big(x_t(x)\big)e^{-\phi_t}dt\bigg], \label{1b}
\end{equation}
$$\mbox{with }\phi_t=\int_0^tc(x_s(x))ds.$$
where $x_t(x)$ is the solution to the It\^o equation
\begin{equation}\label{1aa}
x_t=x+\int_0^t\sigma(x_s)dw_s+\int_0^tb(x_s)ds
\end{equation}
and $\tau=\tau_D(x)$ is the first exit time of $x_t(x)$ from $D$.

If we know a priori that $u\in C^2(D)\cap C(\bar D)$ and $u$ solves (\ref{1a}), then $u$ satisfies (\ref{1b}) via It\^o's formula, which implies the uniqueness of the solution of (\ref{1a}) provided the uniqueness of the solution of (\ref{1aa}). However, in general, $u$ defined by (\ref{1b}) doesn't necessarily have first and second derivatives in the differential operator $L$, especially when the diffusion term $a$ is degenerate, and the differential equation is understood in a generalized sense. We are interested in knowing under what conditions $u$ defined by (\ref{1b}) is twice differentiable and does satisfy (\ref{1a}). 

The accumulation of the research on the existence, uniqueness and regularity of degenerate elliptic or parabolic partial differential equations has become vast. See, for example, H\"{o}rmander \cite{MR0222474}, Kohn-Nirenberg \cite{MR0234118} and Oleinik-Radkevich \cite{MR0457908}, in which analysis techniques for PDEs are used. For probabilistic approaches, we refer to Freidlin \cite{MR833742} and Stroock-Varadhan \cite{MR0387812}, to name a few. 

Our approach, quasiderivative method, is also probabilistic. The concept of quasiderivative was first introduced by N. V. Krylov in \cite{MR965890} (1988), in which this probabilistic technique is applied to find weaker and more flexible conditions on $\sigma$, $b$ and $c$ such that $u$ in (\ref{1b}) is twice continuously differentiable in manifolds without boundary. Since then, this technique has been applied to investigate the smoothness of solutions of various elliptic and parabolic partial differential equations. The first derivatives of various linear elliptic and parabolic PDEs have been estimated under various conditions in Krylov \cite{MR1180379} (1992), \cite{MR1268004} (1993) and \cite{MR2144644} (2004), where each case was treated by its particular choice of quasiderivatives. In Krylov \cite{MR2082053} (2004), a unified quasiderivative method is presented, while $\sigma$ and $b$ are assumed to be constant. As far as the applications to nonlinear equations, for example, in Krylov \cite{MR992979} (1989), derivative estimates  are obtained when controlled diffusion processes and consequently fully nonlinear elliptic equations are considered. %The result in \cite{MR992979} are obtained by reducing controlled processes in domains to those on a surface without boundary in the space having four more dimensions. 

Compared to the operators considered in \cite{MR965890,MR1180379,MR1268004,MR2144644,MR2082053}, the differential equation in this article is more general. The differential operator $L$ in (\ref{1a}) is the general linear elliptic differential operator, and $c$ and $f$ are non-trivial. Also, we estimate the derivatives up to the second order, not just the first order. More presicely, our main target is investigating first derivatives of $u$ if we only assume  $f, g\in C^{0,1}(\bar D)$, as well as the second derivatives therein when assuming $f, g\in C^{1,1}(\bar D)$. Note that, in these cases, one cannot assert that the first and second derivatives of $u$ are bounded up to the boundary (for example, see Remark 1.0.2 and Example 4.2.1 in \cite{MR2144644}). One can only expect to prove that inside $D$ the derivatives of $u$ exist. We show that under our assumptions, the first and second derivatives of $u$ in (\ref{1b}) exist almost everywhere in $D$,  which implies the existence and uniqueness for the Dirichelet problem for the associated linear degenerate elliptic partial differential equation (\ref{1a}) in our setting. We also obtain first and second derivative estimates.

This article is organized as follows: In Section 2, we review the concept of quasiderivative and give an example of it. In Section 3, we take this approach to  show the existence of, and then estimate, the first and second derivatives of $u$ in (\ref{1b}), under the assumption of the non-degeneracy of $a$ with respect to the normal to the boundary and an interior condition to control the moments of quasiderivatives, which is weaker than the nondegeneracy of the diffusion term $a$ and necessary under the aforementioned assumption.

To conclude this section, we introduce the notation: Above we have already defined $C^k(\bar{D}),  k=1 \text{\ or\ } 2$, as the space of bounded continuous and $k$-times continuously differentiable functions in $\bar{D}$ with finite norm given by
\[ |g|_{1,D}=|g|_{0,D}+|g_x|_{0,D},\ \  |g|_{2,D}=|g|_{1,D}+|g_{xx}|_{0,D},\]
respectively, where 
\[|g|_{0,D}=\sup_{x\in D}|g(x)|,\]
$g_x$ is the gradient vector of $g$, and $g_{xx}$ is the Hessian matrix of $g$. For $\alpha\in(0,1]$, the H\"older spaces $C^{k,\alpha}(\bar D)$ are defined as the subspaces of $C^k(\bar D)$ consisting of functions with finite norm
$$|g|_{k,\alpha,D}=|g|_{k,D}+[g]_{\alpha,D},\ \ \mbox{ where }[g]_{\alpha, D}=\sup_{x,y\in D}\frac{|g(x)-g(y)|}{|x-y|^\alpha}.$$

Throughout the article, the summation convention for repeated indices is assumed, and we always put the index in the superscript, since the subscript is for the time variable of stochastic processes. 

We let $\Rd$ is the $d$-dimensional Euclidean space with $x = (x^1, x^2, . . . , x^d )$ representing a typical point in $\Rd$, and $(x, y) = x^iy^i$ is the inner product for $x, y \in \Rd$. For $x,y,z\in\Rd$, set
\begin{align*}
u_{(y)}=&u_{x^i}y^i,\ \  u_{(y)(z)}=u_{x^i x^j}y^i z^j,\ \ u_{(y)}^2=(u_{(y)})^2.
\end{align*}
For any matrix $\sigma=(\sigma^{ij})$,
$$\|\sigma\|^2:=\mathrm{tr}(\sigma\sigma^*).$$
For any $s,t\in\mathbb R$, we define
\[s\wedge t=\min(s,t),\ \  s\vee t=\max(s,t).\]
Constants $K, N$ and $\lambda$ appearing in inequalities are usually not indexed. They may differ even in the same chain of inequalities.

%%%%%%%%%%%%%%%%section 2
\section{Definition and Examples of Quasiderivative}

%In this section, we first introduce the concept of quasiderivative due to N.~V.~ Krylov, then we give an example of quasiderivatives which will be applied in the subsequent sections.

In what follows, we consider the It\^{o} stochastic equation
\begin{equation} 
x_t=x+\int_0^t \sigma^i(x_s)dw_s^i+\int_0^t b(x_s)ds \label{2a}
\end{equation}
on a given complete probablity space $(\Omega,\mathcal{F},P)$, where $x\in\Rd$, $\sigma^i$ and $b$ are (nonrandom) $\Rd$-valued functions with bounded domain $D$ in $\Rd$, defined for $i=1,..., d_1$ with $d_1$ possibly different from $d$, and $w_t:=(w_t^1,..., w_t^{d_1})$ is a $d_1$-dimensional Wiener process with respect to a given increasing filtration $\{\mathcal{F}_t, t\ge 0\}$ of $\sigma$-algebras $\mathcal{F}_t\subset\mathcal{F}$, such that $\mathcal{F}_t$ contain all $P$-null sets. We denote by $\sigma$ the $d\times d_1$ matrix composed of the column-vectors $\sigma^i$, $i=1,..., d_1$. We also assume that $\sigma$ and $b$ are twice continuously differentiable in $\Rd$. Based on the assumptions above, for any $x\in D$, it is known that equation (\ref{2a}) has a unique solution $x_t(x)$ on $[0,\tau(x))$, where
\begin{equation*}
\tau(x)=\inf\{t\ge0 : x_t(x)\notin D\} \qquad (\inf\{\varnothing\}:=\infty).
\end{equation*}

\begin{definition}
We write \[u\in\mathcal{M}^k(D,\sigma,b)\]
if $u$ is a real-valued $k$ times continuously differentiable function given on $\bar{D}$ such that the process $u(x_t(x))$ is a local $\{\mathcal{F}_t\}$-martingale on $[0,\tau(x))$ for any $x\in D$. 
\end{definition}

We abbreviate $\mathcal{M}^k(D,\sigma,b)$ by $\mathcal{M}^k(D)$, or simply $\mathcal{M}^k$ when this will cause no confusion.

\begin{definition} \label{2b}
Let $x\in D$, and let $\gamma$ be a stopping time, such that $\gamma\le\tau(x)$. Assume that $\xi\in\mathbb{R}^d$, $\xi_t$ and $\xi_t^0$ are adapted continuous processes defined on $[0,\gamma]\cap[0,\infty)$ with values in $\mathbb{R}^d$ and $\mathbb{R}$, respectively, such that $\xi_0=\xi$. 

We say that $\xi_t$ is a \textup{\textbf{first quasiderivative}} of $x_t$ in the direction of $\xi$ at point $x$ on $[0,\gamma]$ if for any $u\in\mathcal{M}^1(D,\sigma,b)$ the following process
\begin{equation} \label{2f}
\uxt(x_t(x))+\xi_t^0 u(x_t(x))
\end{equation}
is a local martingale on $[0,\gamma]$. In this case the process $\xi_t^0$ is called a \textup{\textbf{first adjoint process}} for $\xi_t$. If $\gamma=\tau(x)$ we simply say that $\xi_t$ is a first quasiderivative of $x_t(y)$ in D in the direction of $\xi$ at $x$. 
\end{definition}

It is worth mentioning that the first adjoint process is not unique for the first quasiderivative in general. All of the first adjoint processes we consider in this article are local martingales with initial value 0.

\begin{definition} \label{2bb}
Under the assumptions of Definition  \ref{2b}, additionally assume that $\eta\in\mathbb{R}^d$, $\eta_t$ and $\eta_t^0$ are adapted continuous processes defined on $[0,\gamma]\cap[0,\infty)$ with values in $\mathbb{R}^d$ and $\mathbb{R}$, respectively, such that $\eta_0=\eta$. 

We say that $\eta_t$ is a \textup{\textbf{second quasiderivative}} of $x_t$ associated with $\xi_t$ and $\xi_t^0$ in the direction of $\eta$ at point $x$ on $[0,\gamma]$ if for any $u\in\mathcal{M}^2(D,\sigma,b)$ the following process
\begin{equation} \label{2g}
\uxxt(x_t(x))+ \uzt(x_t(x)) + 2\xi_t^0 \uxt(x_t(x))+\eta_t^0 u(x_t(x)),
\end{equation}
$$\mbox{ where }\xi_t\mbox{ and }\xi_t^0\mbox{ are first quasiderivative and first adjoint process.}$$
is a local martingale on $[0,\tau)$. In this case the process $\eta_t^0$ is called a \textup{\textbf{second adjoint process}} for $\eta_t$. If $\gamma=\tau(x)$ we simply say that $\eta_t$ is a second quasiderivative of $x_t(y)$ associated with $\xi_t$ in D in the direction of $\eta$ at $x$. 
\end{definition}

Similarly, the second adjoint process is not unique for the second quasiderivative in general. All second adjoint processes we consider in this article are local martingales with initial value 0.

Now let us consider
$$u(x)=Eg\big(x_\tau(x)\big),$$
that is, we temporarily let $f=c=0$ in (\ref{1b}). Based on the definitions above, if $u\in C^2(\bar{D})$, then the strong Markov property of $x_t(x)$ implies that $u\in\mathcal{M}^2(D)$, and the usual first and second ``derivatives'' with respect to $x$ of the process $x_t(x)$, which are defined as the solutions of the It\^o equations
\begin{equation*}
\xi_t=\xi+\int_0^t\sxs^k(x_s)dw_s^k+\int_0^t\bxs(x_s)ds
\end{equation*}
\begin{equation*}
\eta_t=\eta+\int_0^t\Big[\sxxs^k(x_s)+\szs^k(x_s)\Big]dw_s^k+\int_0^t\Big[\bxxs(x_s)+\bzs(x_s)\Big]ds\
\end{equation*}
are first and second quasiderivatives with zero adjoint processes. This means, the ``quasiderivative'' of a given stochastic process is a generalization of the usual ``derivative'' of the stochastic process.

Now we additionally assume that the domain $D$ is of class $C^2$ with $\partial D$ bounded, $\tau(x)<\infty$ (a.s.), and $g$ is twice continuously differentiable on $\partial D$. We abbreviate $\tau(x)$ to $\tau$. If the process (\ref{2f}) is a uniformly integrable martingale on $[0,\tau]$ and $\xi_{\tau}$ is tangent to $\partial D$ at $x_{\tau}(x)$ (a.s.), then we have
\begin{equation}\label{uxi}
u_{(\xi)}(x)=E[u_{(\xi_\tau)}(x_\tau)+\xi_\tau^0u(x_\tau)]=E[g_{(\xi_\tau)}(x_\tau)+\xi_\tau^0g(x_\tau)].\end{equation}
This shows how we can apply first quasiderivatives to get interior estimates of $u_{(\xi)}$ through $|g|_{1,D}$ or $|g|_{1,\partial D}$.

As far as second derivatives are concerned, first notice that
\[4u_{(\xi)(\zeta)}(x)=u_{(\xi+\zeta)(\xi+\zeta)}(x)-u_{(\xi-\zeta)(\xi-\zeta)}(x).\]
So to estimate $u_{(\xi)(\zeta)}(x), \forall\xi,\zeta\in\Rd$, it suffices to estimate $u_{(\xi)(\xi)}(x), \forall\xi\in\Rd$.

Again, if the process (\ref{2g}) is a uniformly integrable martingale on $[0,\tau]$, $\xi_{\tau}$ and $\eta_{\tau}$ are tangent to $\partial D$ at $x_{\tau}(x)$ (a.s.), then by letting $\eta=0$, we have
\begin{align*}
u_{(\xi)(\xi)}(x)=&u_{(\xi)(\xi)}(x)+u_{(\eta)}(x)\\
=&E[u_{(\xi_\tau)(\xi_\tau)}(x_\tau)+ u_{(\eta_\tau)}(x_\tau)+2\xi_\tau^0 u_{(\xi_\tau)}(x_\tau)+\eta_\tau^0 u(x_\tau)]\\
=&E[g_{(\xi_\tau)(\xi_\tau)}(x_\tau)+u_{(n(x_\tau))}(x_\tau)\cdot h_{(\xi_\tau)(\xi_\tau)}(x_\tau)+ g_{(\eta_\tau)}(x_\tau)\\
&+2\xi_\tau^0 g_{(\xi_\tau)}(x_\tau)+\eta_\tau^0 g(x_\tau)],
\end{align*}
where $n(x)$ is the unit inward normal at $x\in\partial D$ and $h(x) : T_x(\partial D)\rightarrow\mathbb{R}$ is a local representation of $\partial D$ as a graph over tangent space of $\partial D$ at $x$. (Notice that it is different from the first order case that generally $u_{(\xi_\tau)(\xi_\tau)}(x_\tau)\ne g_{(\xi_\tau)(\xi_\tau)}(x_\tau)$.) Since $D$ is of class $C^2$ and $\partial D$ is bounded,
$$h_{(\xi_\tau)(\xi_\tau)}(x_\tau)\le N|\xi_\tau|^2,$$
where $N$ is a positive constant depending on the domain $D$. This shows how we can apply second quasiderivatives to get interior estimates of $u_{(\xi)(\zeta)}$ through $|g|_{2,D}$, or even $|g|_{2,\partial D}$, provided that $u_{(n(y))}(y)$ can be estimated on $\partial D$ in terms of $|g|_{2,D}$ or $|g|_{2,\partial D}$.

It is also worth mentioning that $\eta_{\tau}$ need not be tangent to $\partial D$ at $x_{\tau}(x)$, provided that we can control the moments of $\eta_{t\wedge\tau}$ and estimate the normal derivative of $u$, because we can represent $\eta_{\tau}$ as the sum of the tangential component and the normal component.

The discussion above motivates us on attempting to construct as many quasiderivatives as possible. 

\begin{theorem} \label{2c}
Let $r_t, \hat{r}_t,\pi_t, \hat{\pi}_t, P_t,\hat{P}_t$ be jointly measurable adapted processes with values in $\mathbb{R}$, $\mathbb{R}$, $\mathbb{R}^{d_1}$, $\mathbb{R}^{d_1}$, $\mathrm{Skew}(d_1,\mathbb{R})$, $\mathrm{Skew}(d_1,\mathbb{R})$, respectively, where $\mathrm{Skew}(d_1,\mathbb{R})$ denotes the set of $d_1\times d_1$ skew-symmetric real matrices. Assume that 
\[\int_0^T(|r_t|^4+|\hat{r}_t|^2+|\pi_t|^4+|\hat{\pi}_t|^2+|P_t|^4+|\hat{P}_t|^2)dt<\infty\]
for any $T\in [0,\infty)$. For $x\in D$, $\xi\in\Rd$ and $\eta\in\Rd$, on the time interval $[0,\infty)$, define the processes $\xi_t$ and $\eta_t$ as solutions of the following (linear) equations:
\allowdisplaybreaks\begin{align}
\xi_t=\xi  +  \int_0^t\Big[&\sxs+r_s\sigma+\sigma P_s\Big]dw_s\label{2d}+  \int_0^t\Big[\bxs+2r_sb-\sigma \pi_s\Big]ds,
\end{align}
\begin{equation}\label{2e}
\begin{gathered}
\eta_t=\eta +  \int_0^t\Big[\szs+\hat{r}_s\sigma+\sigma\hat{P}_s +\sxxs +  2r_s\sxs\\
\qquad\qquad\ \ \ \ +2\sxs P_s +2r_s\sigma P_s-r_s^2\sigma+\sigma P_s^2\Big]dw_s\\
\qquad\ \  +  \int_0^t\Big[\bzs+2\hat{r}_s b-\sigma\hat{\pi}_s+\bxxs+4r_s\bxs\\
\qquad\ \ \ -2\sxs\pi_s-2r_s\sigma\pi_s-2\sigma P_s \pi_s\Big]ds,
\end{gathered}
\end{equation}
where in $\sigma, b$ and their derivatives we dropped the argument $x_s(x)$. Also define:
\begin{align}
\xi^0_t=&\int_0^t\pi_sdw_s, \label{2j}\\
\eta^0_t=&(\xi_t^0)^2-\langle\xi^0\rangle_t+\int_0^t\hat{\pi}_sdw_s. \label{2k}
\end{align}
Then $\xi_t$ is a first quasiderivative of $x_t(y)$ in $D$ in the direction of $\xi$ at $x$ and $\xi_t^0$ is a first adjoint process for $\xi_t$, and $\eta_t$ is a second quasiderivative of $x_t(y)$ associated with $\xi_t$ in $D$ in the direction of $\eta$ at $x$ and $\eta_t^0$ is a second adjoint process for $\eta_t$.
\end{theorem}

\begin{remark}
The processes $r_t$ and $\hat r_t$ come from random time change. The processes $\pi_t$ and $\hat\pi_t$ are due to Girsanov's Theorem on changing the probability space, and the processes $P_t$ and $\hat P_t$ are based on changing the Wiener process based on Levy's Theorem.

Equations (\ref{2d}) and (\ref{2e}) give the most general forms of the first and second quasiderivatives known so far. On one hand, they contain various auxiliary processes, $r_t, \pi_t, P_t$, $\hat{r}_t, \hat{\pi}_t, \hat{P}_t,$ which supply us fruitful quasiderivatives for our applications. On the other hand, in specific applications, many of the auxiliary processes are defined to be zero (processes), which make the equations (\ref{2d}) and (\ref{2e}) shorter.
\end{remark}

\begin{proof}

Mimic the proof of Theorem 3.2.1 in \cite{MR2144644} by replacing $y_t(\ve,x)$ as the solution to the It\^o equation
\begin{align*}
\nonumber dy_t=&\sqrt{1+2\ve r_t+\ve^2 \hat{r}_t}\sigma(y_t)e^{\ve P_t}e^{\frac{1}{2} \ve^2\hat{P}_t}dw_t
+\Big[(1+2\ve r_t+\ve^2 \hat{r}_t)b(y_t)\\
&- \sqrt{1+2\ve r_t+\ve^2 \hat{r}_t}\sigma(y_t)e^{\ve P_t}e^{\frac{1}{2} \ve^2
\hat{P}_t} (\ve\pi_t+\frac{1}{2}\ve^2\hat{\pi}_t) \Big]dt
\end{align*}
with initial condition $y=x+\ve\xi+\frac{1}{2}\ve^2\eta$, and then differentiating the local martingale
\begin{equation*}
u(y_t(\ve, x))\exp\big(\int_0^t(\ve\pi_s+\frac{1}{2}\ve^2\hat{\pi}_s)dw_s-\frac{1}{2}\int_0^t|\ve\pi_s+\frac{1}{2}\ve^2\hat{\pi}_s|^2ds\big)
\end{equation*}
twice which turns out to be a local martingale also.

\end{proof}

\begin{remark}
The auxiliary processes $r_t, \pi_t, P_t, \hat{r}_t, \hat{\pi}_t, \hat{P}_t$ are allowed to depend on $\xi_t$ and $\eta_t$. For instance, assume that $r(x,\xi), \pi(x,\xi)$ and $P(x,\xi)$ are locally bounded functions from $D\times\Rd$ to $\mathbb R$, $\mathbb R^{d_1}$ and $\mathrm{Skew}(d_1,\mathbb R)$, respectly, and they are linear with respect to $\xi$. We similarly assume that $\hat r(x,\xi,\eta), \hat\pi(x,\xi,\eta)$ and $\hat P(x,\xi,\eta)$ are locally bounded functions from $D\times\Rd\times\Rd$ to $\mathbb R$, $\mathbb R^{d_1}$ and $\mathrm{Skew}(d_1,\mathbb R)$, respectly, and they are linear with respect to $\eta$. If we define
$$r_t=r(x_t,\xi_t),\qquad \pi_t=\pi(x,\xi), \qquad P_t=P(x_t,\xi_t),$$
$$\hat r_t=r(x_t,\xi_t,\eta_t),\qquad \hat\pi_t=\hat\pi(x,\xi,\eta), \qquad \hat P_t=\hat P(x_t,\xi_t,\eta_t),$$
then the It\^o equations (\ref{2d}) and (\ref{2e}) have unique solutions, since the diffusion term and drift term in both It\^o equations are linear with respect to $\xi_t$ and $\eta_t$, respectively. As a result,  Theorem \ref{2c} still holds. This is exactly how we construct the quasiderivatives in the next section.
\end{remark}

Before ending this section, we introduce two local martingales to be used in applications. 

\begin{theorem}\label{2cc}
Let $c$, $f$, $g$ and $u$ be real-valued twice continuously differentiable functions in $D$. Suppose that $u$ satisfies  \eqref{1a}. Take the processes $r_t, \hat{r}_t,\pi_t$, $\hat{\pi}_t, P_t,\hat{P}_t$, $\xi_t, \eta_t, \xi_t^0, \eta_t^0$ from Theorem \ref{2c}. Then for any $x\in D$, the processes
\begin{equation}\label{Xi}
X_t:=e^{-\phi_t}\Big[u_{(\xi_t)}(x_t)+\tilde{\xi}_t^0u(x_t)\Big]
+\int_0^te^{-\phi_s}\Big[f_{(\xi_s)}(x_s)+\big(2r_s+\tilde{\xi}_s^0\big)f(x_s)\Big]ds,
\end{equation}
\begin{equation}\label{Zeta}
\begin{gathered}
Y_t:=e^{-\phi_t}\Big[u_{(\xi_t)(\xi_t)}(x_t)+u_{(\eta_t)}(x_t)+2\tilde{\xi}_t^0u_{(\xi_t)}(x_t)+\tilde{\eta}_t^0u(x_t)\Big]\\
\qquad+\int_0^te^{-\phi_s}\Big[f_{(\xi_s)(\xi_s)}(x_s)+f_{(\eta_s)}(x_s)+\big(4r_s+2\tilde{\xi}_s^0\big)f_{(\xi_s)}(x_s)\\
+\big(2\hat{r}_s+4\tilde{\xi}_s^0r_s+\tilde{\eta}_s^0\big)f(x_s)\Big]ds,
\end{gathered}
\end{equation}
with
\begin{align*}
&\phi_t:=\int_0^tc(x_s)ds,\\
&\xi_t^{d+1}:=-\int_0^t\big[c_{(\xi_s)}(x_s)+2r_sc(x_s)\big]ds,\\
&\tilde{\xi}_t^0:=\xi_t^0+\xi_t^{d+1},\\
&\eta_t^{d+1}:=-\int_0^t\big[c_{(\xi_s)(\xi_s)}(x_s)+c_{(\eta_s)}(x_s)+4r_sc_{(\xi_s)}(x_s)+2\hat{r}_sc(x_s)\big]ds,\\
&\tilde{\eta}_t^0:=\eta_t^0+2\xi_t^0\xi_t^{d+1}+(\xi_t^{d+1})^2+\eta_t^{d+1},
\end{align*}
are local martingales on $[0,\tau_D(x))$. (We keep writing $x_t$ in place of $x_t(x)$ and drop this argument in many places.)

\end{theorem}

\begin{proof}
Introduce two additional equations
$$
x_t^{d+1}=-\int_0^tc(x_s)ds, \ \ x_t^{d+2}=\int_0^t\exp(x_s^{d+1})f(x_s)ds.
$$
For $\bar{x}=(x,x^{d+1},x^{d+2})\in D\times\mathbb{R}\times\mathbb{R}$, define
$$
\bar{u}(\bar{x})=\exp(x^{d+1})u(x)+x^{d+2}.
$$
It\^o's formula and the assumption that $a^{ij}(x)u_{x^ix^j}+ b^i(x)u_{x^i}-c(x)u+f(x)=0$ in $D$ imply that $\bar{u}(\bar{x}_t(x,0,0))$ is a local martingale on $[0,\tau_D(x))$. That means, $\bar{u}(\bar{x}_t)\in \mathcal{M}^2$.

According to definitions \ref{2b} and \ref{2bb},
$$
\bar{u}_{(\bar{\xi}_t)}(\bar{x}_t)+\xi_t^0\bar{u}(\bar{x}_t) \mbox{  and  } \bar{u}_{(\bar{\eta}_t)}(\bar{x}_t)+\bar{u}_{(\bar{\xi}_t)(\bar{\xi}_t)}(\bar{x}_t)+2\xi_t^0\bar{u}_{(\bar{\xi}_t)}(\bar{x}_t)+\eta_t^0\bar{u}(\bar{x}_t)
$$ 
are local martingales on $[0,\tau_D(x))$, where $\bar{\xi}_t=(\xi_t,\xi_t^{d+1},\xi_t^{d+2})$ and $\bar{\eta}_t=(\eta_t,\eta_t^{d+1},\eta_t^{d+2})$ are first and second quasiderivatives of $\bar{x}_t((x,0,0))$ in the directions of $\bar\xi=(\xi,0,0)$ and $\bar\eta=(\eta,0,0)$, respectively.

Direct computation leads to
\begin{align*}
&\bar{u}_{(\bar{\xi}_t)}(\bar{x}_t)=\exp(x_t^{d+1})\big[u_{(\xi_t)}(x_t)+\xi_t^{d+1}u(x_t)\big]+\xi_t^{d+2},\\
&\bar{u}_{(\bar{\xi}_t)(\bar{\xi}_t)}(\bar{x}_t)=\exp(x_t^{d+1})\big[u_{(\xi_t)(\xi_t)}(x_t)+2\xi_t^{d+1}u_{(\xi_t)}(x_t)+\big(\xi_t^{d+1}\big)^2u(x_t)\big],\\
&\bar{u}_{(\bar{\eta}_t)}(\bar{x}_t)=\exp(x_t^{d+1})\big[u_{(\eta_t)}(x_t)+\eta_t^{d+1}u(x_t)\big]+\eta_t^{d+2},
\end{align*}
with
\allowdisplaybreaks\begin{align*}
\xi_t^{d+2}=&\int_0^t\exp(x_s^{d+1})\big[f_{(\xi_s)}(x_s)+(\xi_s^{d+1}+2r_s)f(x_s)\big]ds,\\
\eta_t^{d+2}=&\int_0^t\exp(x_s^{d+1})\big[f_{(\xi_s)(\xi_s)}(x_s)+f_{(\eta_s)}(x_s)+(2\xi_s^{d+1}+4r_s)f_{(\xi_s)}(x_s)\\
&+\big((\xi_s^{d+1})^2+\eta_s^{d+1}+4r_s\xi_s^{d+1}+2\hat{r}_s\big)f(x_s)\big]ds.
\end{align*}
It remains to notice that $\xi^0_t$ and $\eta^0_t$ are local martingales, so by Lemma II.8.5(c) in \cite{MR1311478}
$$
\xi^0_tx^{d+2}_t-\int_0^t\xi^0_sdx^{d+2}_s, \xi^0_t\xi^{d+2}_t-\int_0^t\xi^0_sd\xi^{d+2}_s \mbox{ and }\eta^0_tx^{d+2}_t-\int_0^t\eta^0_sdx^{d+2}_s
$$
are local martingales.

\end{proof}

%%%%%%%%%%%%%%section 3
\section{Application of quasiderivatives to derivative estimates of non-homogeneous linear degenerate elliptic equations}

In this section, we investigate the smoothness of $u$ given by (\ref{1b}), which is the probabilistic solution of (\ref{1a}).

To be precise, let $\sigma$, $b$ and $c$ in (\ref{1b}) and (\ref{1aa}) be twice continuously differentiable in $\Rd$, and $c$ be non-negative. Let $D\in C^4$ be a bounded domain in $\Rd$, then there exists a function $\psi\in C^4$ satisfying
$$
\psi>0 \mbox{ in }D,\ \ \psi=0\mbox{ and } |\psi_x|\ge1 \mbox{ on }\partial D.
$$
We also assume that
\begin{equation}\label{psi}
L\psi:=a^{ij}(x)\psi_{x^ix^j}+b^i(x)\psi_{x^i}\le-1 \mbox{ in }D.
\end{equation}
%for all $x\in D$, $|y|=1$,
%\begin{equation}
%\begin{gathered}
%\|\sigma(x)\|+\|\sigma_{(y)}(x)\|+\|\sigma_{(y)(y)}(x)\|+|b(x)|+|b_{(y)}(x)|+|b_{(y)(y)}(x)|+|c(x)|+\\
%|c_{(y)}(x)|+|c_{(y)(y)}(x)|+|\psi(x)|+|\psi_{(y)}(x)|+|\psi_{(y)(y)}(x)|+|\psi_{(y)(y)(y)}(x)|\le K_0
%\end{gathered}
%\end{equation}
\begin{equation} \label{3setting}
|\sigma^{ij}|_{2,D}+|b^i|_{2,D}+|c|_{2,D}+|\psi|_{4,D}\le K_0,
\end{equation}
with constant $K_0\in[1,\infty)$.

Let $\mathfrak{B}$ be the set of all skew-symmetric $d_1\times d_1$ matrices. For any positive constant $\lambda$, define
$$D_\lambda=\{x\in D: \psi(x)>\lambda\}.$$

\begin{assumption} \label{nd} (non-degeneracy along the normal to the boundary)
\begin{equation*}\label{nondeg}
(an,n)>0 \mbox{ on }\partial D,
\end{equation*}
where $n$ is the unit normal vector.
\end{assumption}

\begin{assumption} \label{ic} (interior condition to control the moments of quasiderivatives, weaker than the non-degeneracy) There exist functions
\begin{itemize}
\item $\rho(x): D\rightarrow\Rd$, bounded in $D_\lambda$ for all $\lambda>0$; 
\item $Q(x,y): D\times\Rd\rightarrow\mathfrak{B}$, bounded with respect to $x$ in $D_\lambda$ for all $\lambda>0, y\in\Rd$ and linear in $y$;
\item $M(x): D\rightarrow \mathbb R$, bounded in $D_\lambda$ for all $\lambda>0$;
\end{itemize}
such that for any $x\in D$ and $|y|=1$,
\begin{equation}\label{inequality}
\begin{gathered}
\big\|\sigma_{(y)}(x)+(\rho(x),y)\sigma(x)+\sigma(x)Q(x,y)\big\|^2+\\
\ 2\big(y,b_{(y)}(x)+2(\rho(x),y)b(x)\big)\le c(x)+M(x)\big(a(x)y,y\big).
\end{gathered}
\end{equation}
\end{assumption}

Our main result is the following:

\begin{theorem}\label{3c}
Define $u$ by \eqref{1b}, in which $x_t(x)$ is the solution of (\ref{1aa}). Suppose that Assumption \ref{nd} and Assumption \ref{ic} are satisfied.
\begin{enumerate}
\item If $f, g\in C^{0,1}(\bar{D})$, then $u\in C^{0,1}_{loc}(D)$, and for any $\xi\in\Rd$,
\begin{equation}
\big|u_{(\xi)}\big|\le N\bigg(|\xi|+\frac{|\psi_{(\xi)}|}{\psi^{\frac{1}{2}}}\bigg)\big(|f|_{0,1,D}+|g|_{0,1,D}\big)\  a.e.\mbox{ in }D,\label{3d}
\end{equation}
where $N=N(K_0,d,d_1, D)$.
\item If $f, g\in C^{1,1}(\bar{D})$, then $u\in C^{1,1}_{loc}(D)$, and for any $\xi\in\Rd$,
\begin{equation}
\big|u_{(\xi)(\xi)}\big|\le N\bigg(|\xi|^2+\frac{\psi_{(\xi)}^2}{\psi}\bigg)\big(|f|_{1,1,D}+|g|_{1,1,D}\big)\ a.e.\mbox{ in }D\label{3dd},
\end{equation}
where $N=N(K_0,d,d_1, D)$. Furthermore, $u$ is the unique solution in $C_{loc}^{1,1}(D)\cap C^{0,1}(\bar D)$ of
\begin{equation}\label{solva}
\left\{\begin{array}{rcll}
Lu(x)-c(x)u(x)+f(x)&=&0  &\text{a.e. in }D\\ 
u&=&g  &\text{on }\partial D.
\end{array}
\right.
\end{equation}
\end{enumerate}
\end{theorem}

\begin{remark}
The author doesn't know whether the estimates (\ref{3d}) and (\ref{3dd}) are sharp.
\end{remark}

\begin{remark}
We give two examples to show that Assumption \ref{ic} is necessary under Assumption \ref{nd} and how to take advantage of the parameters $\rho, Q, M$ in (\ref{inequality}), respectively. They are similar to Remark V.8.6 and Example VI.1.7 in \cite{MR1311478}. See Example V.8.3, Remark V.8.6, Example VI.1.2 and Example VI.1.7 in \cite{MR1311478} for more details.

In the first example, we take $d=d_1=1$ and $D=(-2,2)$. Let $\sigma(x)=x, b(x)=\beta x$ in $[-2,2]$ and $c(x)=\nu, f(x)=0$ in $[-1,1]$, where $\nu>0, \beta\in\mathbb R$ are constants. Extend $c(x)$ and $f(x)$ outside $[-1,1]$ in such a way that $c(x)\ge\nu, f(x)>0$, and $c$ and $f$ are smooth on $[-2,2]$, bounded and have bounded derivatives up to second order. Let $g(x)=0$ on $\partial D=\{-2,2\}$. Define
$$\tau_1(x)=\inf\{t\ge0:|x_t(x)|\ge1\}, \qquad \tau_2(x)=\inf\{t\ge0:|x_t(x)|\ge2\}.$$

Based on our construction, for all $t\in[0,\tau_2(x)]$ (a.s.),
$$x_t(x)=xe^{w_t+(\beta-1/2)t}.$$
It follows that for any $x\in(0,1]$, $x_t(x)$ takes the value 1 at time $\tau_1(x)$ almost surely. Similarly, for any $x\in[-1,0)$, $x_{\tau_1(x)}(x)=-1$ (a.s.). Also, note that
$$Ee^{-\nu\tau_1(x)}=x^\kappa, \mbox{ with }\kappa=[(\beta-1/2)^2+2\nu]^{1/2}-\beta+1/2.$$
Hence
\begin{equation}
u(x)=\left\{\begin{array}{ll}
Ee^{-\nu\tau_1(x)}u(x_{\tau_1(x)}(x))=u(1)x^\kappa&\mbox{ if }x\in(0,1],\\
Ee^{-\nu\tau_1(x)}u(x_{\tau_1(x)}(x))=u(-1)|x|^\kappa&\mbox{ if }x\in[-1,0),\\
0&\mbox{ if }x=0.
\end{array}
\right.
\end{equation}
Notice that $u(1)>0, u(-1)>0$, so $u(x)$ has Lipschitz continuous derivatives if and only if $\kappa\ge2$. It is equivalent to $1+2\beta\le\nu$, which is exactly (\ref{inequality}) in which $\rho,Q,M$ are vanishing. This example shows that Assumption \ref{ic} is necessary.

Next, we discuss an advantage of the parameters $\rho, Q, M$ in (\ref{inequality}). More precisely, we show that with the help of these parameters, based on some local information, Assumption \ref{ic} holds. Assume that $d=d_1=1$ for the sake of simplicity, and for each $x\in D$ where $\sigma(x)=b(x)=0$, we have
\begin{equation}\label{bds}
|\sigma'(x)|^2+2b'(x)< c(x).
\end{equation}
With this local property, we claim that Assumption \ref{ic} hold. Indeed, we observe that for 
$$\rho(x)=-nb(x),\qquad Q(x,y)=nb(x)y,\qquad M(x)=n,$$
the inequality (\ref{inequality}) becomes
\begin{equation}\label{conv}
|\sigma'(x)|^2+2b'(x)\le c(x)+n\sigma^2(x)+4nb^2(x).
\end{equation}
Suppose that there exists $D_\lambda$, for any $n\in\{1,2,...\}$, there exists a point $x_n$ at which the inequality converse to (\ref{conv}) holds. Then we can exact from the sequence $(\sigma(x_n),\sigma'(x_n),$ $ b(x_n), b'(x_n), c(x_n))$ a subsequence that converges to $(\sigma(x_0),\sigma'(x_0),$ $ b(x_0), b'(x_0), c(x_0))$ for some $x_0\in\bar D_\lambda$. It follows from (\ref{3setting}) that
$$n\sigma^2(x_n)+4nb^2(x_n)<|\sigma'(x_n)|^2+2b'(x_n)\le K_0, \forall n.$$
Therefore, $\sigma(x_0)=b(x_0)=0$ and
$$|\sigma'(x_0)|^2+2b'(x_0)\ge c(x_0)$$
It is a contradiction to (\ref{bds}), so for any $\lambda$, there exists $n_\lambda$, such that the inequality  (\ref{conv}) holds in $D_\lambda$ for  $n_\lambda$. As a consequence, Assumption \ref{ic} is indeed satisfied.  

\end{remark}

The following two remarks are reductions of Theorem \ref{3c}.

%%%%%%%%%%%%%%%reduction 1
\begin{remark}\label{reduction1}
Without loss of generality, we may assume that $c\ge1$ and replace condition (\ref{inequality}) by
\begin{equation}\label{ineq}
\begin{gathered}
\big\|\sigma_{(y)}(x)+(\rho(x),y)\sigma(x)+\sigma(x)Q(x,y)\big\|^2+\\
\ 2\big(y,b_{(y)}(x)+2(\rho(x),y)b(x)\big)\le c(x)-1+M(x)\big(a(x)y,y\big).
\end{gathered}
\end{equation}

Indeed, letting $\displaystyle\tilde{u}=\frac{u}{\psi+1}$ in $D$, we have
$$u_{x^i}=(\psi+1)\tilde{u}_{x^i}+\psi_{x^i}\tilde{u},\ \  u_{x^ix^j}=(\psi+1)\tilde{u}_{x^ix^j}+\psi_{x^j}\tilde{u}_{x^i}+\psi_{x^i}\tilde{u}_{x^j}+\psi_{x^ix^j}\tilde{u}$$
Hence (\ref{1a}) turns into
$$
\left\{\begin{array}{rcll}
\displaystyle\tilde a^{ij}(x)\tilde{u}_{x^ix^j}+\tilde b^i(x)\tilde{u}_{x^i}-\tilde{c}(x)\tilde{u}+f(x)&=&0,&\mbox{ in }D\\
\displaystyle\tilde{u}&=&\tilde g:=g/(1+\psi), &\mbox{ on }\partial D
\end{array}
\right.
$$
with
$$
\tilde{a}^{ij}=(\psi+1)a^{ij},\ \ \tilde{b}^i=2a^{ij}\psi_{x^j}+(\psi+1)b^i,\ \ \tilde{c}=-L\psi+(1+\psi)c.
$$
Notice that $\tilde{\sigma}^{ij}=\sqrt{\psi+1}\sigma^{ij}$. So a direct computation implies that
$$
|\tilde\sigma^{ij}|_{2,D}+|\tilde b^i|_{2,D}+|\tilde c|_{2,D}+|\psi|_{4,D}\le (d^2+2d+2)K_0^3,
$$
which plays the same role as (\ref{3setting}). 

Since $L\psi\le-1$ and $c\ge0$, $\tilde c\ge1$.

We also have $(\tilde{a}n,n)>0$ on $\partial D$. Under the substitutions on $\sigma$, $b$ and $c$, by inequality (\ref{inequality}), we have
$$
\begin{gathered}
\frac{1}{\psi+1}\Big\|\tilde\sigma_{(y)}(x)-\frac{1}{2}\frac{\psi_{(y)}}{\psi+1}\tilde\sigma(x)+\big(\rho(x),y\big)\tilde\sigma(x)+\tilde\sigma(x)Q(x,y)\Big\|^2\\
+\frac{2}{\psi+1}\Big(y,\tilde{b}_{(y)}(x)-\frac{\psi_{(y)}}{\psi+1}\tilde{b}(x)+2\big(\rho(x),y\big)\tilde{b}(x)\Big)\\
\le \frac{\tilde{c}(x)+L\psi}{\psi+1}+M(x)\big(a(x)y,y\big)+2\bigg(y,\Big(\frac{2a\psi_x}{\psi+1}\Big)_{(y)}+2\big(\rho(x),y\big)\frac{2a\psi_x}{\psi+1}\bigg).
\end{gathered}
$$
Collecting similar terms and noticing that $L\psi\le-1$, we get
$$
\begin{gathered}
\Big\|\tilde\sigma_{(y)}(x)+\big(\tilde \rho(x),y\big)\tilde\sigma(x)+\tilde\sigma(x)Q(x,y)\Big\|^2+2\Big(y,\tilde b_{(y)}(x)+2\big(\tilde \rho(x),y\big)\tilde b(x)\Big)\\
\le \tilde c(x)-1+\tilde M(x)\big(\tilde a(x)y,y\big)+4\big(\tilde a_{(y)}(x)\psi_x,y\big),
\end{gathered}
$$
with
$$
\tilde \rho(x):=\rho(x)-\frac{\psi_x}{2(\psi+1)},
$$
and $\tilde M(x)$ is in terms of  $M(x),K_0$ and $|\rho(x)|$.

The term $4(\tilde a_{(y)}\psi_x,y)$ can not be bounded by $\tilde M(x)(\tilde a(x)y,y)$. However, notice that 
$$\tilde a_{(y)}(x)=\tilde\sigma(x)\tilde\sigma^*_{(y)}(x).$$
So $\tilde M(x)(\tilde a(x)y,y)+4(\tilde a_{(y)}\psi_x,y)$ can be rewritten in the form of 
$$\bigg(\tilde \sigma(x)\Big(\frac{\tilde M(x)}{2}\tilde\sigma^*(x)y+4\tilde\sigma^*_{(y)}(x)\psi_x\Big),y\bigg),$$ 
which can play the same role as that of $M(x)(a(x)y,y)$, which, in the proof, will be rewritten in the form of
$$ \Big(\sigma(x)\cdot\frac{M(x)}{2}\sigma^*(x)y,y\Big).$$

A direct computation shows that if $\tilde u$ satisfies estimates (\ref{3d}) and (\ref{3dd}), we have the same estimates for $u$.

\begin{comment}
Suppose we have estimates (\ref{3d}) and (\ref{3dd}) for $\tilde u$, then we can conclude that
\allowdisplaybreaks\begin{align*}
|u_{(\xi)}(x)|=&|\psi_{(\xi)}(x)\tilde u(x)+(1+\psi(x))\tilde u_{(\xi)}(x)|\\
\le&N|\xi|(|\tilde g|_{0,D}+E\int_0^\tau e^{-t}dt\cdot|f|_{0,D})+N(|\xi|+\frac{|\psi_{(\xi)}|}{\psi^{\frac{1}{2}}})(|f|_{1,D}+|\tilde g|_{1,D})\\
\le&N(|\xi|+\frac{|\psi_{(\xi)}|}{\psi^{\frac{1}{2}}})(|f|_{1,D}+|g|_{1,D})\\
|u_{(\xi)(\xi)}(x)|=&|\psi_{(\xi)(\xi)}(x)\tilde u(x)+2\psi_{(\xi)}(x)\tilde u_{(\xi)}(x)+(1+\psi(x))\tilde u_{(\xi)(\xi)}(x)|\\
\le&N|\xi|^2(|\tilde g|_{0,D}+|f|_{0,D})+N|\xi|(|\xi|+\frac{|\psi_{(\xi)}|}{\psi^{\frac{1}{2}}})(|f|_{1,D}+|\tilde g|_{1,D})\\
&+N(|\xi|^2+\frac{\psi_{(\xi)}^2}{\psi})(|f|_{2,D}+|\tilde g|_{2,D})\\
\le&N(|\xi|^2+\frac{\psi_{(\xi)}^2}{\psi})(|f|_{2,D}+|g|_{2,D})
\end{align*}
That means, $u$ also satisfies the same derivative estimates.
\end{comment}
\end{remark}

%%%%%%%%%%%%%%%%reduction 2

\begin{remark}\label{reduction2}
Without loss of generality, we may assume that $u\in C^1({D})$ and $f,g\in C^1(\bar D)$ when investigating first derivatives of $u$, and $u\in C^2({D})$ and $f,g\in C^2(\bar D)$ when investigating second derivatives of $u$. 

Let us take the first situation for example, in which $u, f, g$ can be assumed to be of class $C^1$. The second situation can be discussed by almost the same argument. 

We define the process $x_t^\ve(x)$ to be the solution to the equation
$$x_t=x_0+\int_0^t\sigma(x_s)dw_s+\int_0^t\ve I d\tilde w_s+\int_0^tb(x_s)ds$$
where $\tilde w_t$ is a $d$-dimensional Wiener process independent of $w_t$ and $I$ is the identity matrix of size $d\times d$, and we define $\tau^\ve(x)$ to be the first exit time of $x_t^\ve(x)$ from $D$, then for the function 
$$u^\ve(x):=E\bigg[g\big(x^\ve_{\tau^\ve(x)}(x)\big)e^{-\phi^\ve_{\tau^\ve(x)}}+\int_0^{\tau^\ve(x)}f\big(x^\ve_t(x)\big)e^{-\phi^\ve_t}dt\bigg],$$
$$\mbox{with }\phi^\ve_t:=\int_0^tc(x_t^\ve(x))dt,$$
the relation $u^\ve\rightarrow u$ holds as $\ve\rightarrow0$. Indeed, notice that
\allowdisplaybreaks\begin{align*}
E|g(x^\ve_{\tau^\ve}(x))-g(x_\tau(x))|
\le& KE\Big(|x^\ve_{\tau^\ve\wedge\tau}(x)-x_{\tau^\ve\wedge\tau}(x)|\\
&+(\tau^\ve\vee\tau-\tau^\ve\wedge\tau)+(\tau^\ve\vee\tau-\tau^\ve\wedge\tau)^{1/2}\Big),\\
E|e^{-\phi^\ve_{\tau^\ve}}-e^{-\phi_\tau}|\le& Ee^{-\tau^\ve\wedge\tau}|\phi^\ve_{\tau^\ve}-\phi_\tau|\\
\le& KEe^{-\tau^\ve\wedge\tau}\Big(\tau^\ve\wedge\tau\cdot\sup_{t\le\tau^\ve\wedge\tau}|x_t^\ve(x)-x_t(x)|\\
&+(\tau^\ve\vee\tau-\tau^\ve\wedge\tau)+(\tau^\ve\vee\tau-\tau^\ve\wedge\tau)^{1/2}\Big)\\
\le&KE\Big(\sup_{t\le\tau^\ve\wedge\tau}|x_t^\ve(x)-x_t(x)|\\
&+(\tau^\ve\vee\tau-\tau^\ve\wedge\tau)+(\tau^\ve\vee\tau-\tau^\ve\wedge\tau)^{1/2}\Big),
\end{align*}
and
\begin{align*}
&E\Big|\int_0^{\tau^\ve}f(x^\ve_t(x))e^{-\phi_t^\ve}dt-\int_0^\tau f(x_t(x))e^{-\phi_t}dt\Big|\\
\le&E\int_0^{\tau^\ve\wedge\tau}|f(x^\ve_t(x))e^{-\phi_t^\ve}-f(x_t(x))e^{-\phi_t}|dt+KE(\tau^\ve\vee\tau-\tau^\ve\wedge\tau)\\
\le&E\int_0^{\tau^\ve\wedge\tau}K\Big(|x_t^\ve(x)-x_t(x)|+t\cdot\sup_{s\le t}|x_s^\ve(x)-x_s(x)|\Big)e^{-t}dt\\
&+KE(\tau^\ve\vee\tau-\tau^\ve\wedge\tau)\\
\le&KE\Big(\sup_{t\le\tau^\ve\wedge\tau}|x_t^\ve(x)-x_t(x)|+(\tau^\ve\vee\tau-\tau^\ve\wedge\tau)\Big),
\end{align*}
where $K$ is a constant depending on $|g|_{0,1,D}, |f|_{0,1,D}$ and $K_0$. It follows that
\begin{align*}
|u^\ve(x)-u(x)|\le &KE\Big(\sup_{t\le\tau^\ve\wedge\tau}|x_t^\ve(x)-x_t(x)|\\
&+(\tau^\ve\vee\tau-\tau^\ve\wedge\tau)+(\tau^\ve\vee\tau-\tau^\ve\wedge\tau)^{1/2}\Big)\\
\le &K\Big(E\sup_{t\le\tau^\ve\wedge\tau\wedge T}|x_t^\ve(x)-x_t(x)|+KP(\tau>T)\\
&+EI_1+EI_2+\sqrt{EI_1}+\sqrt{EI_2}\Big),
\end{align*}
where
$$I_1=(\tau^\ve\vee\tau-\tau^\ve\wedge\tau)I_{\tau>\tau^\ve}=(\tau-\tau^\ve)I_{\tau>\tau^\ve},$$
$$I_2=(\tau^\ve\vee\tau-\tau^\ve\wedge\tau)I_{\tau<\tau^\ve}=(\tau^\ve-\tau)I_{\tau<\tau^\ve}.$$

It remains to notice that
$$E\sup_{t\le \tau^\ve\wedge\tau\wedge T}|x_t(x)-x^\ve_t(x)|\le e^{KT}\ve\rightarrow0, \mbox{ as }\ve\rightarrow 0,$$
$$P(\tau>T)\le\frac{E\tau}{T}\le \frac{1}{T}E\int_0^\tau\Big(-L\psi\big(x_t(x)\big)\Big)dt=\frac{\psi(x)-\psi(x_\tau(x))}{T}\le\frac{K_0}{T},$$
and
\allowdisplaybreaks\begin{align*}
E(\tau-\tau^\ve)I_{\tau>\tau^\ve}=&E\int_{\tau\wedge\tau^\ve}^{\tau}1dt\\
\le&-E\int_{\tau\wedge\tau^\ve}^{\tau}L\psi(x_t(x))dt\\
=&-E\Big(\psi\big(x_{\tau}(x)\big)-\psi\big(x_{\tau^\ve}(x)\big)\Big)I_{\tau^\ve<\tau}\\
=&E\psi\big(x_{\tau^\ve(x)}(x)\big)
I_{\tau^\ve<\tau}\\
=&E\Big(\psi\big(x_{\tau^\ve}(x)\big)-\psi\big(x^\ve_{\tau^\ve}(x)\big)\Big)
I_{\tau^\ve<\tau}\\
\le&E\Big(\psi\big(x_{\tau^\ve}(x)\big)-\psi\big(x^\ve_{\tau^\ve}(x)\big)\Big)
I_{\tau^\ve<\tau\le T}+2K_0P(\tau>T)\\
\le&K_0E\sup_{t\le \tau^\ve\wedge\tau\wedge T}|x_t(x)-x^\ve_t(x)|+\frac{2K_0^2}{T}\\
E(\tau^\ve-\tau)I_{\tau<\tau^\ve}\le&-2E\int_{\tau\wedge\tau^\ve}^{\tau^\ve}L^\ve\psi(x^\ve_t(x))dt\\
\le&\cdots \le K_0E\sup_{t\le \tau^\ve\wedge\tau\wedge T}|x_t(x)-x^\ve_t(x)|+\frac{2K_0^2}{T}.
\end{align*}
Hence by first letting $\ve\downarrow 0$ and then $T\uparrow\infty$, we conclude that
$$|u^\ve(x)-u(x)|\rightarrow0\mbox{ as }\ve\rightarrow 0.$$

Moreover, for small $\ve$ the condition (\ref{psi}) holds for $2\psi$, taken instead of $\psi$ and $L^\ve$ associated to the process $x_t^\ve(x)$. The matrix $\sigma^\ve$ corresponding to the process $x_t^\ve(x)$ is obtained by attaching the identity matrix, multiplied by $\ve$, to the right of the original matrix $\sigma$. In this connection we modify $P(x,y)$ by adding zero entries on the right and below to form a $(d_1+d)\times(d_1+d)$ matrix. Then the condition (\ref{ineq}) corresponding to the process $x_t^\ve(x)$ will differ from the original condition by the fact that the term $\ve^2(\rho(x),y)^2d$ appears on the left, and $\frac{1}{2}M(x)\ve^2$ on the right. From this it is clear that the condition (\ref{ineq}) for the process $x_t^\ve(x)$ (for all $\ve$) also holds when $M(x)$ is replaced by $M(x)+2|\rho(x)|^2d$.

Finally, from analysis of PDE, we know that for $\ve\ne0$ the nondegenerate elliptic equation $L^\ve w=0$ in $D$ with the boundary condition $w=g$ on $\partial D$ has a solution that is continuous in $\bar D$ and twice continuously differentiable in $D$, and $u^\ve=w$ in $D$ by It\^o's formula. From this it follows that it suffices to prove the theorem for small $\ve\ne0$, the process $x^\ve_t(x)$, and a function $u^\ve$ that is continuously differentiable in $D$. Of course, we must be sure that the constants $N$ in (\ref{3d}) is chosen to be independent of $\ve$, which is true as we can see in the proof of the theorem. Observing further that for each fixed $\ve\ne0$ the functions $f$ and $g$ can be uniformly approximated in $\bar D$ by infinitely differentiable functions, in such a way that the last factor in (\ref{3d}) increases by at most a factor of two when $f$ and $g$ are replaced by the approximating functions, while for the latter the function $w$ (i.e., $u^\ve$) has continuous and bounded first derivatives in $\bar D$, we conclude that we may assume $u$ has continuous first derivatives in $D$ and $f,g\in C^1(\bar D)$ when investigating first derivatives of $u$.

\end{remark}

Before proving the theorem, let us prove four lemmas. In Lemma \ref{lemma1} we estimate the first exit time. It is a well-known result, but we still prove for the sake of completeness. Lemma \ref{lemma4} concerns the estimate of the first derivative along the normal to the boundary, to be used when estimating the second derivatives. In Lemma \ref{3l1} and Lemma \ref{3l2}, we construct two supermartingales, which will play the roles of barriers near the boundary and in the interior of the domain, respectively.

%%%%%%%%%%%%%%%%%lemma 1
\begin{lemma} \label{lemma1}
Let $\tau_{D_0}(x)$ be the first exit time of $x_t(x)$ from $D_0$, which is a sub-domain of $D$ containing $x$. Then we have
$$E\tau_{D_0}(x)\le E\tau_{D}(x)\le\psi(x)\le |\psi|_{0,D},$$
$$E\tau^2_{D_0}(x)\le E\tau^2_{D}(x)\le2|\psi|_{0,D}\psi(x)\le 2|\psi|_{0,D}^2.$$
\end{lemma}

\begin{proof}
The fact that $D_0\subset D$ implies $E\tau_{D_0}(x)\le E\tau_{D}(x)$ and $E\tau^2_{D_0}(x)\le E\tau^2_{D}(x)$. Now we abbreviate $\tau_{D}(x)$ by $\tau(x)$, or simply $\tau$ when this will cause no confusion. By (\ref{psi}) and It\^o's formula, we have
\begin{align*}
E\tau&=E\int_0^\tau1dt\le-E\int_0^\tau L\psi dt=\psi(x)-E\psi(x_\tau)=\psi(x),\\
E\tau^2&=2E\int_0^\infty(\tau-t)I_{\tau>t}dt=2E\int_0^\infty I_{\tau>t}E\tau(x_t)dt\\
&\le 2\sup_{y\in D}E\tau(y)\cdot E\int_0^\infty I_{\tau>t}dt=2\sup_{y\in D}E\tau(y)\cdot E\tau\le2|\psi|_{0,D}\psi(x).
\end{align*}
\end{proof}

%%%%%%%%%%%%%%%%lemma 4
\begin{lemma}\label{lemma4}
If $f, g\in C^2(\bar D)$, and $u\in C^1(\bar D)$, then for any $y\in\partial D$ we have
\begin{equation}
|u_{(n)}(y)|\le K(|g|_{2,D}+|f|_{0,D})\label{normal},
\end{equation}
where $n$ is the unit inward normal on $\partial D$ and the constant $K$ depends only on $K_0$.
\end{lemma}

\begin{proof}
Fix a $y\in\partial D$, and choose $\ve_0>0$ so that $y+\ve n\in D$ as long as $0<\ve\le\ve_0$. Also, fix an $\ve\in(0,\ve_0]$ and let $x:=y+\ve n$.
By It\^o's formula, 
$$d(g(x_t)e^{-\phi_t})=e^{-\phi_t}g_{(\sigma^k)}(x_t)dw_t^k+e^{-\phi_t}(Lg(x_t)-c(x_t)g(x_t))dt.$$
Notice that
$$E\int_0^\infty\Big(e^{-\phi_t}g_{(\sigma^k)}(x_t)\Big)^2I_{t\le\tau}dt\le N|g|_{1,D}^2E\tau<\infty.$$
The Wald identities hold:
$$E\int_0^\tau e^{-\phi_t}g_{(\sigma^k)}(x_t)dw_t^k=0.$$
Thus
$$Ee^{-\phi_\tau}g\big(x_\tau(x)\big)=g(x)+E\int_0^\tau e^{-\phi_t}\big(Lg(x_t)-c(x_t)g(x_t)\big)dt.$$
Together with (\ref{1b}), we have
\begin{align*}
u(x)=&g(x)+E\int_0^\tau e^{-\phi_t}(Lg(x_t(x))-c(x_t(x))g(x_t(x)))dt\\
&+E\int_0^{\tau}f\big(x_t(x)\big)e^{-\phi_t}dt\\
\le&g(x)+(|Lg|_{0,D}+|c|_{0,D}|g|_{0,D}+|f|_{0,D})E\tau\\
\le&g(x)+K(|g|_{2,D}+|f|_{0,D})\psi(x).
\end{align*}
Notice that $u(y)=g(y)$ and $\psi(y)=0$. So we have
$$\frac{u(y+\ve n)-u(y)}{\ve}\le\frac{g(y+\ve n)-g(y)}{\ve}+K(|g|_{2,D}+|f|_{0,D})\frac{\psi(y+\ve n)-\psi(y)}{\ve}.$$
Letting $\ve\downarrow0$, we get
$$u_{(n)}(y)\le K(|g|_{2,D}+|f|_{0,D}).$$
Replacing $u$ with $-u$ yields the same estimate of $(-u)_{(n)}$ from above, which is an estimate of $u_{(n)}$ from below. Combining the estimates from above and from below leads to ($\ref{normal}$) and proves the lemma.
\end{proof}

For constants $\delta$ and $\lambda$, such that $0<\delta<\lambda^2<\lambda<1$, define
\begin{align*}
D^\lambda&=\{x\in D: \psi<\lambda\},\\
D_\delta^\lambda&=\{x\in D: \delta<\psi<\lambda\}.
\end{align*}

Considering that the formulas of the quasiderivatives $\xi_t,\eta_t$ and the barrier functions $\mathrm B_1(x,\xi), \mathrm B_2(x,\xi)$ constructed in Lemmas \ref{3l1} and \ref{3l2} are complicated, and the proofs of Lemmas \ref{3l1} and \ref{3l2} are long and technical, we first make a remark on the motivation of these constructions.

\begin{remark}
As discussed right after Definition \ref{2bb} in section 2, when investigating the first derivative of $u$, the main difficulty comes from the term $Eu_{(\xi_\tau)}(x_\tau)$ in (\ref{uxi}), and we should try to construct $\xi_t$ in such a way that $\xi_{\tau}$ is tangent to $\partial D$ at $x_{\tau}(x)$ almost surely. Considering that the diffusion process $x_t$ and domain $D$ are quite general in our setting, it is almost impossible, since there is no way to know when or where $x_t$ exits the domain. Therefore, what we actually try is constructing $\xi_t$ in such a way that either $\xi_{\tau}$ is tangent to $\partial D$ at $x_{\tau}(x)$ almost surely, or $|u_{(\xi_\tau)}(x_\tau)|$ is bounded by a nonnegative local supermartingale $\mathrm B(x_\tau,\xi_\tau)$. If we succeed, we will have
\begin{equation*}
E|u_{(\xi_\tau)}(x_\tau)|
\left\{\
\begin{array}{ll}
=E|g_{(\xi_\tau)}(x_\tau)|\le |g|_{1,\partial D}E|\xi_\tau|&\mbox{if $\xi_{\tau(x)}$ is tangent to $\partial D$}\\
\le E\mathrm B(x_\tau,\xi_\tau)\le \mathrm B(x,\xi) &\mbox{if $\xi_{\tau(x)}$ is not tangent to $\partial D$}.
\end{array}
\right.
\end{equation*}
As we will see in the following two lemmas, $\mathrm B(x,\xi)=\sqrt{\mathrm B_1(x,\xi)}$ near the boundary, while $B(x,\xi)=\sqrt{\mathrm B_2(x,\xi)}$ in the interior of the domain.

\end{remark}
%%%%%%%%%%%%%%lemma5
\begin{lemma} \label{3l1}
Introduce
$$
\varphi(x)=\lambda^2+\psi(1-\frac{1}{4\lambda}\psi), \ \ \mathrm{B}_1(x,\xi)=\big[\lambda+\sqrt{\psi}(1+\sqrt{\psi})\big]|\xi|^2+K_1\varphi^\frac{3}{2}\pxsop,
$$
where $K_1\in[1,\infty)$ is a constant depending only on $K_0$.

In $D^\lambda$, if we construct first and second quasiderivatives by (\ref{2d}) and (\ref{2e}), in which
\allowdisplaybreaks\begin{align*}
&r(x,\xi):=\rho(x,\xi)+\frac{\psi_{(\xi)}}{\psi},\ \ r_t:=r(x_t,\xi_t),\\
&\mbox{ where }\rho(x,\xi):=-\frac{1}{A}\sum_{k=1}^{d_1}\psk(\psk)_{(\xi)},\mbox{ with }A:=\sum_{k=1}^{d_1}\psk^2;\\
&\hat{r}(x,\xi):=\pxsops,\ \ \hat{r}_t:=\hat{r}(x_t,\xi_t);\\
&\pi^k(x,\xi):=\frac{2\psk\psi_{(\xi)}}{\varphi\psi},\ \  k=1,...,d_1,\ \ \pi_t:=\pi(x_t,\xi_t);\\
&P^{ik}(x,\xi):=\frac{1}{A}\big[\psk(\mypsi)_{(\xi)}-\mypsi(\psk)_{(\xi)}\big],\ \  i,k=1,...,d_1,\ \ P_t:=P(x_t,\xi_t); \\
&\hat{\pi}_t^k=\hat{P}_t^{ik}=0,\ \ \forall i,k=1,...d_1,\forall t\in[0,\infty).
\end{align*}
Then for sufficiently small $\lambda$, when $x_0\in D_\delta^\lambda$, $\xi_0\in\Rd$ and $\eta_0=0$, we have
\begin{enumerate}
\item $\mathrm{B}_1(x_t,\xi_t)$ and $\sqrt{\mathrm{B}_1(x_t,\xi_t)}$ are local supermartingales on $[0,\tau_1^\delta]$, where $\tau_1^\delta=\tau_{D_\delta^\lambda}(x_0)$;
\item $\displaystyle{E\int_0^{\tau_1^\delta}|\xi_t|^2+\pxtsops dt\le N\mathrm{B}_1(x_0,\xi_0)}$;
\item $\displaystyle{E\sup_{t\le\tau_1^\delta}|\xi_t|^2\le N\mathrm{B}_1(x_0,\xi_0)}$;
\item $\displaystyle{E|\eta_{\tau_1^\delta}|\le E\sup_{t\le\tau_1^\delta}|\eta_t|\le N\mathrm{B}_1(x_0,\xi_0)}$;
\item $\displaystyle{E\Big(\int_0^{\tau_1^\delta}|\eta_t|^2 dt\Big)^{\frac{1}{2}}\le N\mathrm{B}_1(x_0,\xi_0)}$;
\end{enumerate}
where $N$ is a constant depending on $K_0$ and $\lambda$.
\end{lemma}

\begin{proof}

Throughout the proof, keep in mind that the constant $K$ depend only on $K_0$, while the constants $N\in[1,\infty)$ and $\lambda_0\in(0,1)$ depend on $K_0$ and $\lambda$.

First, notice that, on $\partial D$, we have
$$
A=\sum_{k=1}^{d_1}\psk^2=2(a\psi_x,\psi_x)=2|\psi_x|(an,n)\ge 2\delta,
$$
where the constant $\delta>0$, because of the compactness of $\partial D$. Replacing $\psi$ by $\psi/2\delta$ if needed, we may, therefore, assume that $A\ge1$.

By It\^o's formula, for $t<\tau_1^\delta$, we have
%$$d\psi=\mypsi dw_t^i+L\psi dt$$
$$d\pxt=[(\mypsi)_{(\xi_t)}+ r_t\mypsi+\psk  P_t^{ki}]dw_t^i+[(L\psi)_{(\xi_t)}+2 r_tL\psi-\mypsi\pi^i_t]dt.$$
A crucial fact about this equation is that owing to our choice of $r$ and $P$
\[(\mypsi)_{(\xi_t)}+ r_t\mypsi+\psk  P^{ki}_t=\frac{\pxt}{\psi}\mypsi.\]
Thus
\begin{equation}
d\pxt=\frac{\pxt}{\psi}\mypsi dw_t^i+[(L\psi)_{(\xi_t)}+2 r_tL\psi-\mypsi\pi_t^i]dt.\label{4l}
\end{equation}
Let 
\[  \bar\sigma:=\sigma_{(\xi)}+ r\sigma+\sigma P,\qquad\bar b:=b_{(\xi)}+2 r b.\]
We have
\begin{equation}\label{sb1}
\|\bar\sigma\|\le K(|\xi|+\frac{|\px|}{\psi}),
\end{equation}
\begin{equation}\label{sb2}
|\bar b|\le K(|\xi|+\frac{|\px|}{\psi}).
\end{equation}
\begin{comment}
Again, by It\^o's formula, we have
\allowdisplaybreaks\begin{align*}
&d|\xi_t|^2=2(\xi_t,\bar\sigma^k)dw_t^k+[2(\xi_t,\bar b)+\|\bar\sigma\|^2]dt\\
&d\sqrt{\psi}=\frac{\psk}{2\sqrt{\psi}}dw_t^k+[\frac{L\psi}{2\sqrt{\psi}}-\frac{A}{8\psi^{\frac{3}{2}}}]\\
&d\varphi=(1-\frac{\psi}{2\lambda})\psk d w_t^k+[(1-\frac{\psi}{2\lambda})L\psi-\frac{A}{4\lambda}]dt\\
&d\varphi^\frac{3}{2}=\frac{3}{2}\varphi^\frac{1}{2}(1-\frac{\psi}{2\lambda})\psk dw_t^k+\Big\{\frac{3}{2}\varphi^\frac{1}{2}\big[(1-\frac{\psi}{2\lambda})L\psi-\frac{A}{4\lambda}\big]+\frac{3}{8}\varphi^{-\frac{1}{2}}(1-\frac{\psi}{2\lambda})^2A\Big\}dt\\
&d\psi_{(\xi_t)}^2=2\pxtsop\psk dw_t^k+\Big\{2\pxt\big[(L\psi)_{(\xi_t)}+2 r_tL\psi-\psk\pi^k_t\big]+\pxtsops A\Big\}dt\nonumber\\
&d\psi^{-1}=-\psi^{-2}\psk dw_t^k+[-\psi^{-2}L\psi+\frac{A}{\psi^3}]dt\\
&d\frac{\psi_{(\xi_t)}^2}{\psi}=\pxtsops\psk dw_t^k+\Big\{2\frac{\psi_{(\xi_t)}}{\psi}\big[(L\psi)_{(\xi_t)}+2r_tL\psi-\psk\pi^k_t\big]-\pxtsops L\psi\Big\}dt\\
&\qquad\ \ =\pxtsops\psk dw_t^k+\Big\{2\frac{\psi_{(\xi_t)}}{\psi}\big[(L\psi)_{(\xi_t)}+2\rho_tL\psi\big]+3\pxtsops L\psi-\frac{4}{\varphi}A\pxtsops\Big\}dt
\end{align*}
\end{comment}
By It\^o's formula,
\begin{equation}
d\mathrm{B}_1(x_t, \xi_t)=\Gamma_1(x_t, \xi_t)dt+\Lambda_1^k(x_t,\xi_t) dw_t^k
\end{equation}
with
\begin{align*}
\Gamma_1(x,\xi)=I_1+I_2+...+I_{13}
\end{align*}
\begin{comment} 
\allowdisplaybreaks\begin{align*}
\Gamma_1(x, \xi)=&\sqrt{\psi}(1+\sqrt{\psi})[2(\xi,\bar{b})+\|\bar\sigma\|^2]+(1+2\sqrt{\psi})|\xi|^2[\frac{L\psi}{2\sqrt{\psi}}-\frac{A}{8\psi^\frac{3}{2}}]\\
&+\frac{A}{4\psi}|\xi|^2+(1+2\sqrt{\psi})\frac{\psk}{\sqrt{\psi}}(\xi,\bar\sigma^k)\\
&+K_1\Big\{\frac{3}{2}\varphi^\frac{1}{2}\big[(1-\frac{\psi}{2\lambda})L\psi-\frac{A}{4\lambda}\big]+\frac{3}{8}\varphi^{-\frac{1}{2}}(1-\frac{\psi}{2\lambda})^2A\Big\}\pxsop\\
&+K_1\varphi^{\frac{3}{2}}\Big\{2\frac{\psi_{(\xi)}}{\psi}\big[(L\psi)_{(\xi)}+2\rho L\psi\big]+3\pxsops L\psi-\frac{4}{\varphi}A\pxsops\Big\}\\
&+K_1\frac{3}{2}\varphi^\frac{1}{2}(1-\frac{\psi}{2\lambda})A\pxsops
\end{align*}
\end{comment}
where
\allowdisplaybreaks\begin{align*}
&I_1=\lambda[2(\xi,\bar b)+\|\bar\sigma\|^2]\le\lambda K(|\xi|^2+\pxsops)\le K\lambda^{\frac{5}{2}}\frac{|\xi|^2}{\psi^\frac{3}{2}}+K\varphi^{\frac{1}{2}}\pxsops,\\
&\mbox{here we apply (\ref{sb1}), (\ref{sb2}) and } \lambda\le\varphi^{\frac{1}{2}},\\
&I_2=-\lambda2(\xi,\sigma^k)\pi^k\le  \frac{K\lambda|\xi||\psi_{(\xi)}|}{\varphi\psi}\le\frac{\lambda^2|\xi|^2}{32\cdot2^{\frac{3}{2}}\varphi^{\frac{5}{2}}}+\frac{K\varphi^{\frac{1}{2}}\psi_{(\xi)}^2}{\psi^2}\\
&\qquad\le\frac{|\xi|^2}{32\cdot2^{\frac{3}{2}}\varphi^{\frac{3}{2}}}+\frac{K\varphi^{\frac{1}{2}}\psi_{(\xi)}^2}{\psi^2}\le\frac{|\xi|^2}{32\psi^\frac{3}{2}}+K\varphi^\frac{1}{2}\pxsops,\\
&\mbox{here we apply }\lambda^2\le\varphi, \mbox{ and then observe that }\psi\le2\varphi,\\
&I_3=\sqrt{\psi}(1+\sqrt{\psi})2(\xi,\bar b)\le\sqrt{\psi}K|\xi|(|\xi|+\frac{|\psi_{(\xi)}|}{\psi})\le  K\lambda\frac{|\xi|^2}{\psi^\frac{3}{2}},\\
&\mbox{ here we apply (\ref{sb2})},\\
&I_4=-\sqrt{\psi}(1+\sqrt{\psi})2(\xi,\sigma^k)\pi^k\le \frac{K\sqrt{\psi}|\xi||\psi_{(\xi)}|}{\varphi\psi}\le\frac{\psi|\xi|^2}{32\cdot2^{\frac{5}{2}}\varphi^{\frac{5}{2}}}+\frac{K\varphi^{\frac{1}{2}}\psi_{(\xi)}^2}{\psi^2}\\
&\qquad\le\frac{|\xi|^2}{32\psi^\frac{3}{2}}+K\varphi^\frac{1}{2}\pxsops,\\
&\mbox{ here we observe that }\psi\le2\varphi,\\
&I_5=\sqrt{\psi}(1+\sqrt{\psi})\|\bar\sigma\|^2\le K\sqrt{\psi}(|\xi|^2+\pxsops)\le K\lambda^2\frac{|\xi|^2}{\psi^\frac{3}{2}}+K\varphi^\frac{1}{2}\pxsops,\\
&\mbox{ here we apply (\ref{sb1})},\\
&I_6=(1+2\sqrt{\psi})|\xi|^2[\frac{L\psi}{2\sqrt{\psi}}-\frac{A}{8\psi^\frac{3}{2}}]\le-\frac{|\xi|^2}{8\psi^\frac{3}{2}},\\
&I_7=\frac{A}{4\psi}|\xi|^2\le K\frac{|\xi|^2}{\psi}\le K\sqrt{\lambda}\frac{|\xi|^2}{\psi^\frac{3}{2}},\\
&I_8=(1+2\sqrt{\psi})\frac{\psk}{\sqrt{\psi}}(\xi,\bar\sigma^k)\le K\frac{|\xi|}{\sqrt{\psi}}(|\xi|+\frac{|\psi_{(\xi)}|}{\psi})\le K\lambda\frac{|\xi|^2}{\psi^\frac{3}{2}}+\frac{|\xi|^2}{32\psi^\frac{3}{2}}+K\varphi^\frac{1}{2}\frac{\psi_{(\xi)}^2}{\psi^2},\\
&\mbox{ here we apply (\ref{sb1}) and }\psi\le2\varphi,\\
&I_9=K_1\frac{3}{2}\varphi^\frac{1}{2}\big[(1-\frac{\psi}{2\lambda})L\psi-\frac{A}{4\lambda}\big]\pxsop+K_1\varphi^{\frac{3}{2}}3\pxsops L\psi\le 0,\\
&I_{10}=K_1\frac{3}{8}\varphi^{-\frac{1}{2}}(1-\frac{\psi}{2\lambda})^2A\pxsop\le K_1\frac{3}{8}\frac{\psi}{\varphi^\frac{1}{2}}A\pxsops\le K_1\frac{3}{4}\varphi^\frac{1}{2}A\pxsops,\\
&\mbox{ here we use }\psi\le2\varphi,\\
&I_{11}=K_1\varphi^{\frac{3}{2}}2\frac{\psi_{(\xi)}}{\psi}\big[(L\psi)_{(\xi)}+2\rho L\psi\big]\le K_1K\varphi^{\frac{3}{2}}\frac{|\psi_{(\xi)}|}{\psi}|\xi|\le K_1K\lambda\varphi^{\frac{1}{2}}\frac{|\psi_{(\xi)}|}{\psi}|\xi|\\
&\qquad\le K_1\lambda\varphi^\frac{1}{2}\pxsops+K_1K\lambda^3\frac{|\xi|^2}{\psi^\frac{3}{2}},\\
&\mbox{ here we first notice that }\varphi\le2\lambda,\mbox{ and then apply }\psi\le2\varphi,\\
&I_{12}=-K_1\varphi^\frac{3}{2}\frac{4}{\varphi}A\pxsops=-K_14\varphi^\frac{1}{2}A\pxsops,\\
&I_{13}=K_1\frac{3}{2}\varphi^\frac{1}{2}(1-\frac{\psi}{2\lambda})A\pxsops\le K_1\frac{3}{2}\varphi^\frac{1}{2}A\pxsops.
\end{align*}
Collecting our estimates above we see that, when $x\in D_\delta^\lambda$,
\begin{align*}
\Gamma_1(x,\xi)\le& \bigg[K(\lambda^{\frac{5}{2}}+\sqrt\lambda)+K_1K\lambda^3+\Big(\frac{3}{32}-\frac{1}{8}\Big)\bigg]\frac{|\xi|^2}{\psi^\frac{3}{2}}\\
&+\bigg[K+K_1\lambda+K_1A\Big(\frac{3}{4}+\frac{3}{2}-4\Big)\bigg]\varphi^\frac{1}{2}\pxsops.
\end{align*}
Recall that $K$ and $K_1$ depend only on $K_0$. By first choosing $K_1$ such that $K_1\ge K$, then letting $\lambda$ be sufficiently small, we get
\begin{equation}
\Gamma_1(x,\xi)\le -\frac{1}{64}\frac{|\xi|^2}{\psi^\frac{3}{2}}-\frac{1}{2}\varphi^\frac{1}{2}\pxsops\le-\frac{1}{64\lambda^\frac{3}{2}}|\xi|^2-\frac{\lambda}{2}\pxsops\le0\label{3e}.
\end{equation}
It follows that $\mathrm{B}_1(x_t,\xi_t)$ is a local supermartingale on $[0,\tau_1^\delta]$.

Also, notice that $f(x)=\sqrt{x}$ is concave, so $\sqrt{\mathrm{B}_1(x_t,\xi_t)}$ is a local supermartingale on $[0,\tau_1^\delta]$. Thus (1) is proved.

From (\ref{3e}), there exists a sufficiently small positive $\lambda_0$, such that
$$
\Gamma_1(x,\xi)+\lambda_0(|\xi|^2+\pxsops)\le 0,\forall x\in D_\delta^\lambda.
$$
Therefore, 
\begin{align*}
\lambda_0 E\int_0^{\tau_1^\delta}\bigg(|\xi_t|^2+\pxtsops \bigg)dt\le&-E\int_0^{\tau_1^\delta}\Gamma_1(x_t,\xi_t)\\
=&\mathrm{B}_1(x_0,\xi_0)-E\mathrm{B}_1(x_{\tau_1^\delta},\xi_{\tau_1^\delta})\le\mathrm{B}_1(x_0,\xi_0),
\end{align*}
which proves (2).

Since
$$
|\xi_t|^2=|\xi_0|^2+\int_0^t2(\xi_s,\bar b)+\|\bar\sigma\|^2ds+\int_0^t2(\xi_s,\bar\sigma)dw_s,
$$
by Burkholder-Davis-Gundy inequality, for $\tau_n=\tau_1^\delta\wedge\inf\{t\ge0:|\xi_t|\ge n\}$, we have,
\begin{align*}
E\sup_{t\le\tau_n}|\xi_t|^2\le& |\xi_0|^2+\int_0^{\tau_n}\Big(2|\xi_t|\cdot|\bar b|+\|\bar\sigma\|^2\Big)dt+6E\Big(\int_0^{\tau_n}|(\xi_t,\bar\sigma)|^2dt\Big)^{\frac{1}{2}}\\
\le&|\xi_0|^2+NE\int_0^{\tau_n}\bigg(|\xi_t|^2+\frac{\psi_{(\xi_t)}^2}{\psi^2}\bigg)dt+E\Big(\int_0^{\tau_n} N|\xi_t|^2(|\xi_t|^2+\frac{\psi_{(\xi_t)}^2}{\psi^2})dt\Big)^{\frac{1}{2}}\\
\le&N\mathrm{B}_1(x_0,\xi_0)+E\Big[\sup_{t\le\tau_n}|\xi_t|\cdot\Big(\int_0^{\tau_n} N(|\xi_t|^2+\frac{\psi_{(\xi_t)}^2}{\psi^2})dt\Big)^{\frac{1}{2}}\Big]\\
\le&N\mathrm{B}_1(x_0,\xi_0)+\frac{1}{2}E\sup_{t\le\tau_n}|\xi_t|^2+\frac{1}{2}E\Big(\int_0^{\tau_n} N(|\xi_t|^2+\frac{\psi_{(\xi_t)}^2}{\psi^2})dt\Big)\\
\le&N\mathrm{B}_1(x_0,\xi_0)+\frac{1}{2}E\sup_{t\le\tau_n}|\xi_t|^2,
\end{align*}
which implies that
$$
E\sup_{t\le\tau_n}|\xi_t|^2\le N\mathrm{B}_1(x_0,\xi_0).
$$
Now  (3) is obtained by letting $n\rightarrow\infty$.

Now we estimate the moments of second quasiderivative $\eta_t$. Based on our definition, we have
$$
d\eta_t=[\sigma_{(\eta_t)}+G(x_t,\xi_t)]dw_t+[b_{(\eta_t)}+H(x_t,\xi_t)]dt,
$$
with
\begin{align*}
&G(x,\xi)=\sigma_{(\xi)(\xi)}+2r\sigma_{(\xi)}+(2\sigma_{(\xi)}+2r\sigma+\sigma P)P+(\hat{r}-r^2)\sigma,\\
&H(x,\xi)=b_{(\xi)(\xi)}+4rb_{(\xi)}+2\pxsops b.
\end{align*}
Therefore, we have the estimates
\begin{equation*}
\|G\|\le N|\xi|(|\xi|+\frac{|\psi_{(\xi)}|}{\psi}),\qquad|H|\le N(|\xi|^2+\pxsops).
\end{equation*}
It\^o's formula implies
\begin{align*}
%&d(e^{2\varphi})=2e^{2\varphi}\Big[(1-\frac{\psi}{2\lambda})\psk\Big]dw_t^k+2e^{2\varphi}\Big[(1-\frac{\psi}{2\lambda})L\psi-\frac{A}{4\lambda}+(1-\frac{\psi}{2\lambda})^2A\Big]dt\\
&d(|\eta_t|^2e^{2\varphi})=\theta(x_t,\xi_t,\eta_t)dt+\mu^k(x_t,\xi_t,\eta_t)dw_t^k,
\end{align*}
where
\begin{align*}
\theta(x,\xi,\eta)=e^{2\varphi}\Big\{&2|\eta|^2\big[(1-\frac{\psi}{2\lambda})L\psi-\frac{A}{4\lambda}+(1-\frac{\psi}{2\lambda})^2A\big]+\|\sigma_{(\eta)}+G(x,\xi)\|^2\\
&+2(\eta,b_{(\eta)}+H(x,\xi))+2(\eta,\sigma_{(\eta)}+G(x,\xi))[(1-\frac{\psi}{2\lambda})\psk]\Big\}.
\end{align*}
It is not hard to see that, for any $x\in D_\delta^\lambda$,
\begin{align*}
\theta(x,\xi,\eta)\le e^{2\varphi}\Big\{2(1-\frac{1}{4\lambda})A|\eta|^2+N\big[|\eta|^2+|\xi|^2(|\xi|^2+\pxsops)+|\eta|(|\xi|^2+\pxsops)\big]\Big\}.
\end{align*}
So for sufficiently small $\lambda$, we have
$$
\theta(x,\xi,\eta)+\lambda_0|\eta|^2\le Ne^{2\varphi}(|\xi|^2+|\eta|)(|\xi|^2+\pxsops).
$$
Then for any bounded stopping time $\gamma$ with respect to $\{\mathcal{F}_t\}$, we have
$$
E(e^{2\varphi}|\eta_{\gamma}|^2)+\lambda_0 E\int_0^{\gamma}|\eta_t|^2dt\le E\int_0^{\gamma}Ne^{2\varphi}(|\xi_t|^2+|\eta_t|)(|\xi_t|^2+\frac{\psi_{(\xi_t)}^2}{\psi^2})dt.
$$
Let $\tau_n=\tau_1^\delta\wedge\inf\{t\ge0:e^{\varphi}|\eta_t|\ge n\}$. Recall that $\eta_0=0$. By Theorem III.6.8 in \cite{MR1311478}, we have
\begin{align*}
&E\sup_{t\le\tau_n}(e^{\varphi}|\eta_t|)\\
\le& 3E\Big(\int_0^{\tau_n}Ne^{2\varphi}(|\xi_t|^2+|\eta_t|)(|\xi_t|^2+\frac{\psi_{(\xi_t)}^2}{\psi^2})dt\Big)^{\frac{1}{2}}\\
\le& E\Big[\Big(\int_0^{\tau_n}9Ne^{2\varphi}|\xi_t|^2(|\xi_t|^2+\frac{\psi_{(\xi_t)}^2}{\psi^2})dt\Big)^{\frac{1}{2}}+\Big(\int_0^{\tau_n}9Ne^{2\varphi}|\eta_t|(|\xi_t|^2+\frac{\psi_{(\xi_t)}^2}{\psi^2})dt\Big)^{\frac{1}{2}}\Big]\\
\le& E\Big[N\sup_{t\le\tau_n}|\xi_t|\cdot\Big(\int_0^{\tau_n}|\xi_t|^2+\frac{\psi_{(\xi_t)}^2}{\psi^2}dt\Big)^{\frac{1}{2}}+\sup_{t\le\tau_n}\sqrt{e^{\varphi}|\eta_t|}\cdot\Big(\int_0^{\tau_n}N(|\xi_t|^2+\frac{\psi_{(\xi_t)}^2}{\psi^2})dt\Big)^{\frac{1}{2}}\Big]\\
\le& NE\sup_{t\le\tau_n}|\xi_t|^2+NE\int_0^{\tau_n}\bigg(|\xi_t|^2+\frac{\psi_{(\xi_t)}^2}{\psi^2}\bigg)dt+\frac{1}{2}E\sup_{t\le\tau_n}(e^{\varphi}|\eta_t|)\\
\le&\frac{1}{2}E\sup_{t\le\tau_n}(e^{\varphi}|\eta_t|)+N\mathrm{B}_1(x_0,\xi_0),
\end{align*}
which implies that
$$E\sup_{t\le\tau_n}|\eta_t|\le E\sup_{t\le\tau_n}(e^{\varphi}|\eta_t|)\le N\mathrm{B}_1(x_0,\xi_0),$$
$$E\Big(\int_0^{\tau_n}|\eta_t|^2 dt\Big)^\frac{1}{2}\le N\mathrm{B}_1(x_0,\xi_0).$$
Letting $n\rightarrow\infty$, we conclude that (4) and (5) are true.
\end{proof}

%%%%%%%%%%%%%%%lemma6
\begin{lemma} \label{3l2}
Introduce
$$
\mathrm{B}_2(x,\xi)=\lambda^\frac{3}{4}|\xi|^2.
$$

If we construct first and second quasiderivatives by (\ref{2d}) and (\ref{2e}), in which
\allowdisplaybreaks\begin{align*}
&r(x,y):=(\rho(x),y),\ \ r_t:=r(x_t,\xi_t),\ \ \hat{r}_t:=r(x_t,\eta_t),\\
&\pi(x,y):=\frac{M(x)}{2}\sigma^*(x)y,\ \ \pi_t:=\pi(x_t,\xi_t),\ \ \hat{\pi}_t:=\pi(x_t,\eta_t),\\
&P(x,y):=Q(x,y),\ \ P_t:=P(x_t,\xi_t),\ \ \hat{P}_t:=P(x_t,\eta_t).
\end{align*}
Then for sufficiently small $\lambda$, when $x_0\in D_{\lambda^2}$, $\xi_0\in\Rd$ and $\eta_0=0$, we have
\begin{enumerate}
\item $e^{-\phi_t}\mathrm{B}_2(x_t,\xi_t)$ and $\sqrt{e^{-\phi_t}\mathrm{B}_2(x_t,\xi_t)}$ are local supermartingales on $[0,\tau_2)$, where $\tau_2=\tau_{D_{\lambda^2}}(x)$;
\item $\displaystyle{E\int_0^{\tau_2} e^{-\phi_t}|\xi_t|^2 dt\le N\mathrm{B}_2(x_0,\xi_0)}$;
\item $\displaystyle{E\sup_{t\le\tau_2}e^{-\phi_t}|\xi_t|^2\le N\mathrm{B}_2(x_0,\xi_0)}$;
\item $\displaystyle{Ee^{-\phi_{\tau_2}}|\eta_{\tau_2}|\le E\sup_{t\le\tau_2}e^{-\phi_t}|\eta_t|\le N\mathrm{B}_2(x_0,\xi_0)}$;
\item $\displaystyle{E\Big(\int_0^{\tau_2} e^{-2\phi_t}|\eta_t|^2 dt\Big)^\frac{1}{2}\le N\mathrm{B}_2(x_0,\xi_0)}$;
\item The above inequalities are still all true if we replace $\phi_t$ by $\phi_t-\frac{1}{2}t$. More precisely, we have
$$E\int_0^{\tau_2} e^{-\phi_t+\frac{1}{2}t}|\xi_t|^2 dt\le N\mathrm{B}_2(x_0,\xi_0),\ E\sup_{t\le\tau_2}e^{-\phi_t+\frac{1}{2}t}|\xi_t|^2\le N\mathrm{B}_2(x_0,\xi_0),$$
$$E\Big(\int_0^{\tau_2} e^{-2\phi_t+t}|\eta_t|^2 dt\Big)^\frac{1}{2}\le N\mathrm{B}_2(x_0,\xi_0),\ E\sup_{t\le\tau_2}e^{-\phi_t+\frac{1}{2}t}|\eta_t|\le N\mathrm{B}_2(x_0,\xi_0),$$
\end{enumerate}
where $N$ is constant depending on $K_0$ and $\lambda$.
\end{lemma}

\begin{proof}
First of all, replacing $K_0$ by
$$
\max\bigg\{K_0,\sup_{x\in D_{\lambda^2}}|\rho(x)|,\sup_{x\in D_{\lambda^2}, y\in \Rd}\frac{\|Q(x,y)\|}{|y|},\sup_{x\in D_{\lambda^2}}M(x)\bigg\},
$$
we may assume that
$$
\sup_{x\in D_{\lambda^2}}|\rho(x)|\le K_0,\ \ \sup_{x\in D_{\lambda^2}}\|Q(x,y)\|\le K_0|y|,\forall y\in \Rd,\ \  \sup_{x\in D_{\lambda^2}}M(x)\le K_0.
$$

By It\^o's formula, for $t<\tau_2$, we have
$$
d|\xi_t|^2=\Lambda^k_2(x_t,\xi_t)dw_t^k+\Gamma_2(x_t,\xi_t)dt,
$$
where
$$
\Lambda_2(x,\xi)=2(\xi_t,\sxt+r_t\sigma+\sigma P_t),
$$
\allowdisplaybreaks\begin{equation*}
\Gamma_2(x,\xi)=\big[2(\xi,b_{(\xi)}+2rb-\sigma\pi)+\|\sigma_{(\xi)}+r\sigma+\sigma P\|^2\big]\le (c-1)|\xi|^2.
\end{equation*}

So
\begin{equation}\label{bd}
\begin{gathered}
d\big(e^{-\phi_t}|\xi_t|^2\big)=e^{-\phi_t}\big[\Gamma_2(x_t,\xi_t)-c(x_t)|\xi_t|^2\big]dt+dm_t\\
\le-e^{-\phi_t}|\xi_t|^2dt+dm_t,
\end{gathered}
\end{equation}
where $m_t$ is a local martingale.

Thus $e^{-\phi_t}\mathrm{B}_2(x_t,\xi_t)$ is a local supermartingale on $[0,\tau_2)$.

Also, notice that $f(x)=\sqrt{x}$ is concave, so $\sqrt{e^{-\phi_t}\mathrm{B}_2(x_t,\xi_t)}$ is a local supermartingale on $[0,\tau_2]$. (1) is proved.

From (\ref{bd}), we also have
\begin{equation*}
E\int_0^{\tau_2} e^{-\phi_t}|\xi_t|^2 dt=\mathrm{B}_2(x_0,\xi_0)-Ee^{-\phi_{\tau_2}}\mathrm{B}_2(x_{\tau_2},\xi_{\tau_2})\le\mathrm{B}_2(x_0,\xi_0),
\end{equation*}
which proves (2).

Since
$$
e^{-\phi_t}|\xi_t|^2=|\xi_0|^2+\int_0^te^{-\phi_s}\big[2(\xi_s,\bar b)+\|\bar\sigma\|^2-c|\xi_t|^2\big]ds+\int_0^te^{-\phi_s}2(\xi_s,\bar\sigma)dw_s,
$$
by Burkholder-Davis-Gundy inequality, for $\tau_n=\tau_2\wedge\inf\{t\ge0:|\xi_t|\ge n\}$, we have,
\begin{align*}
E\sup_{t\le\tau_n}e^{-\phi_t}|\xi_t|^2\le& |\xi_0|^2+\int_0^{\tau_n}e^{-\phi_t}\big[2|\xi_t|\cdot|\bar b|+\|\bar\sigma\|^2+c|\xi_t|^2\big]dt\\
&+12E\Big(\int_0^{\tau_n}e^{-2\phi_t}|(\xi_t,\bar\sigma)|^2dt\Big)^{\frac{1}{2}}\\
\le&|\xi_0|^2+NE\int_0^{\tau_n}e^{-\phi_t}|\xi_t|^2dt+E\Big(\int_0^{\tau_n} Ne^{-2\phi_t}|\xi_t|^4dt\Big)^{\frac{1}{2}}\\
\le&N\mathrm{B}_2(x_0,\xi_0)+E\Big[\sup_{t\le\tau_n}e^{-\frac{1}{2}\phi_t}|\xi_t|\cdot\Big(\int_0^{\tau_n} Ne^{-\phi_t}|\xi_t|^2dt\Big)^{\frac{1}{2}}\Big]\\
\le&N\mathrm{B}_2(x_0,\xi_0)+\frac{1}{2}E\sup_{t\le\tau_n}e^{-\phi_t}|\xi_t|^2+\frac{1}{2}E\Big(\int_0^{\tau_n} Ne^{-\phi_t}|\xi_t|^2dt\Big)\\
\le&N\mathrm{B}_2(x_0,\xi_0)+\frac{1}{2}E\sup_{t\le\tau_n}e^{-\phi_t}|\xi_t|^2,
\end{align*}
which implies that
$$
E\sup_{t\le\tau_n}e^{-\phi_t}|\xi_t|^2\le N\mathrm{B}_2(x_0,\xi_0).
$$
So (3) is true by letting $n\rightarrow\infty$.

Now we estimate the moments of the second quasiderivative $\eta_t$. Based on our definition, we have
$$
d\eta_t=[\tilde\sigma+G]dw_t+[\tilde b+H]dt,
$$
where
\begin{align*}
&\tilde\sigma=\bar\sigma(x,\eta)=\sigma_{(\eta)}+\hat r\sigma+\sigma\hat{P},\\
&\tilde b=\bar b(x,\eta)=b_{(\eta)}+2\hat r b-\sigma\hat\pi,\\
&G=G(x,\xi)=\sigma_{(\xi)(\xi)}+2r\sigma_{(\xi)}-r^2\sigma+(2\sigma_{(\xi)}+2r\sigma+\sigma P)P,\\
&H=H(x,\xi)=b_{(\xi)(\xi)}+4rb_{(\xi)}-2(\sigma_{(\xi)}+r\sigma-\sigma P)\pi.
\end{align*}
From the expressions above, we have the estimates
\begin{equation*}
\|G\|\le N|\xi|^2,\qquad|H|\le N|\xi|^2.
\end{equation*}
Hence It\^o's formula implies
$$
d\big(e^{-2\phi_t}|\eta_t|^2\big)=e^{-2\phi_t}\big[2(\eta_t,\tilde b+H)+\|\tilde\sigma+G\|^2-2c|\eta_t|^2\big]dt+2e^{-2\phi_t}(\eta_t,\tilde\sigma+G)dw_t^k.
$$
Notice that
\begin{align*}
&2(\eta,\tilde b+H)+\|\tilde\sigma+G\|^2-2c|\eta|^2\\
=&2(\eta,\tilde b)+\|\tilde\sigma\|^2-2c|\eta|^2+2(\eta, H)+|H|^2+2(\tilde\sigma^k,G^k)\\
\le&(c-1)|\eta|^2-2c|\eta|^2+|\eta|^2+N|\xi|^4\\
\le&-|\eta|^2+N|\xi|^4.
\end{align*}
So for any bounded stopping time $\gamma$ with respect to $\{\mathcal{F}_t\}$, we have
$$
Ee^{-2\phi_\gamma}|\eta_\gamma|^2+E\int_0^{\gamma}e^{-2\phi_t}|\eta_t|^2dt\le E\int_0^{\gamma}Ne^{-2\phi_t}|\xi_t|^4dt.
$$
Recall that $\eta_0=0$. By Theorem III.6.8 in \cite{MR1311478}, we have
\begin{align*}
E\sup_{t\le\tau_2}e^{-\phi_t}|\eta_t| \le& 3E\Big(\int_0^{\tau_2}Ne^{-2\phi_t}|\xi_t|^4dt\Big)^{\frac{1}{2}}\\
\le& E\Big[\sup_{t\le\tau_2}e^{-\frac{1}{2}\phi_t}|\xi_t|\cdot\Big(\int_0^{\tau_2}9Ne^{-\phi_t}|\xi_t|^2dt\Big)^{\frac{1}{2}}\Big]\\
\le& \frac{1}{2}E\sup_{t\le\tau_2}e^{-\phi_t}|\xi_t|^2+\frac{1}{2}E\int_0^{\tau_2}9Ne^{-\phi_t}|\xi_t|^2dt\\
\le&N\mathrm{B}_2(x_0,\xi_0),
\end{align*}
$$E\Big(\int_0^{\tau_2}e^{-2\phi_t}|\eta_t|^2dt\Big)^\frac{1}{2}\le3E\Big(\int_0^{\tau_2}Ne^{-2\phi_t}|\xi_t|^4dt\Big)^{\frac{1}{2}}\le N\mathrm{B}_2(x_0,\xi_0),$$
which implies that (4) and (5) are true.

Finally, rewritting $c-1$ by $(c-\frac{1}{2})-\frac{1}{2}$ and repeating the argument above, we conclude that (6) is true.

\end{proof}

Now we are ready to prove the theorem.

%%%%%%%%%%%%%%%proof of theorem
\begin{proof}[Proof of (\ref{3d})] 

Denote $\tau_{D_\delta^\lambda}(x)$ and $\tau_{D_{\lambda^2}}(x)$ by $\tau_1^\delta$ and $\tau_2$, respectively.

From (\ref{1b}) we immediately have
\begin{equation}\label{u}
|u|_{0,D}\le|g|_{0,D}+|f|_{0,D}E\int_0^\tau e^{-t}dt\le|g|_{0,D}+|f|_{0,D}.
\end{equation}

When $x_0\in D_\delta^\lambda$, by Theorem \ref{2cc}, we have
%\begin{align*}
%u_{(\xi_0)}(x_0)=E\bigg\{&e^{-\phi_{\tau_1^\delta}}\Big[u_{(\xi_{\tau_1^\delta})}(x_{\tau_1^\delta})+(\xi_{\tau_1^\delta}^0+\xi_{\tau_1^\delta}^{d+1})u(x_{\tau_1^\delta})\Big]\\
%&+\int_0^{\tau_1^\delta}e^{-\phi_s}\Big[f_{(\xi_s)}(x_s)+\big(2r_s+|\xi_s^0|+\xi_s^{d+1}\big)f(x_s)\Big]ds\bigg\}
%\end{align*}
$$u_{(\xi_0)}(x_0)=X_0=EX_{\tau_1^\delta}.$$
So from (\ref{Xi}) and (\ref{u}),
\begin{align*}
|u_{(\xi_0)}(x_0)|\le &E\Big|u_{(\xi_{\tau_1^\delta})}(x_{\tau_1^\delta})+(\xi_{\tau_1^\delta}^0+\xi_{\tau_1^\delta}^{d+1})u(x_{\tau_1^\delta})\Big|\\
&+|f|_{1,D}E\int_0^{\tau_1^\delta}e^{-s}\Big(|\xi_s|+2r_s+|\xi_s^0|+|\xi_s^{d+1}|\Big)ds\\
\le &E\Big|u_{(\xi_{\tau_1^\delta})}(x_{\tau_1^\delta})\Big|+\Big(|g|_{0,D}+|f|_{0,D}\Big)\Big(E|\xi_{\tau_1^\delta}^0|+E|\xi_{\tau_1^\delta}^{d+1}|\Big)\\
&+|f|_{1,D}\Big(E\int_0^{\tau_1^\delta}|\xi_s|+2r_sds+E\sup_{t\le\tau_1^\delta}|\xi_t^{0}|+E\sup_{t\le\tau_1^\delta}|\xi_t^{d+1}|\Big).
\end{align*}
By Lemma \ref{3l1}, Davis inequality and H\"older inequality,
\allowdisplaybreaks\begin{align}
\nonumber E|u_{(\xi_{\tau_1^\delta})}(x_{\tau_1^\delta})|\le&\sup_{ x\in \partial D_\delta^\lambda}\uxosbw\cdot E\sqrt{\mathrm{B}_1(x_{\tau_1^\delta},\xi_{\tau_1^\delta})}\\
\nonumber\le&\sup_{x\in \partial D_\delta^\lambda}\uxosbw\cdot \sbwxzxz,\\
E|\xi_{\tau_1^\delta}^0|\le E\sup_{t\le\tau_1^\delta}|\xi_t^{0}|\le&3E\langle\xi^0\rangle_{\tau_1^\delta}^\frac{1}{2}\le3\Big(E\langle\xi^0\rangle_{\tau_1^\delta}\Big)^\frac{1}{2}\\
\nonumber\le &N\bigg(E\int_0^{\tau_1^\delta}\frac{\psi_{(\xi_s)}^2}{\psi^2}ds\bigg)^\frac{1}{2}\le N\sbwxzxz,\\
E|\xi_{\tau_1^\delta}^{d+1}|\le E\sup_{t\le\tau_1^\delta}|\xi_t^{d+1}|\le& NE\int_0^{\tau_1^\delta}|\xi_s|+\frac{|\psi_{(\xi_s)}|}{\psi}ds\\
\nonumber\le&NE\int_0^\infty I_{s\le\tau_1^\delta}\cdot I_{s\le\tau_1^\delta}\Big(|\xi_s|+\frac{|\psi_{(\xi_s)}|}{\psi}\Big)ds\\
\nonumber\le&N\bigg(E\tau_1^\delta\bigg)^{\frac{1}{2}}\bigg(E\int_0^{\tau_1^\delta}\Big(|\xi_s|^2+\frac{\psi_{(\xi_s)}^2}{\psi^2}\Big)ds\bigg)^\frac{1}{2}\\
\nonumber\le&N\sbwxzxz,\\
E\int_0^{\tau_1^\delta}(|\xi_s|+2r_s)ds\le&NE\int_0^{\tau_1^\delta}\bigg(|\xi_s|+\frac{|\psi_{(\xi_s)}|}{\psi}\bigg)ds\le N\sbwxzxz.
\end{align}
Collecting all estimates above, we conclude that
$$
\uxzxz\le\sup_{x\in \partial D_\delta^\lambda}\uxosbw\cdot \sbwxzxz+N(|g|_{0,D}+|f|_{1,D})\sbwxzxz.
$$
So for any $x_0\in D_\delta^\lambda$, $\xi_0\in\Rd\setminus\{0\}$, we have
\begin{equation}\label{3h}
\frac{\uxzxz}{\sbwxzxz}\le\sup_{x\in \partial D_\delta^\lambda}\uxosbw+N_1,
\end{equation}
with
\begin{equation}\label{N1}
N_1=N(|g|_{1,D}+|f|_{1,D}).
\end{equation}
Similarly, when $x_0\in D_{\lambda^2}$, by Theorem \ref{2cc}, we have
%\begin{align*}
%u_{(\xi_0)}(x_0)=E\bigg\{&e^{-\phi_{\tau_2}}\Big[u_{(\xi_{\tau_2})}(x_{\tau_2})+(\xi^0_{\tau_2}+\xi_{\tau_2}^{d+1})u(x_{\tau_2})\Big]\\
%&+\int_0^{\tau_2}e^{-\phi_s}\Big[f_{(\xi_s)}(x_s)+\big(2r_s+\xi_s^0+\xi_s^{d+1}\big)f(x_s)\Big]ds\bigg\}
%\end{align*}
$$u_{(\xi_0)}(x_0)=X_0=EX_{\tau_2}.$$
Again, from (\ref{Xi}) and (\ref{u}),
\begin{align*}
|u_{(\xi_0)}(x_0)|\le &Ee^{-\phi_{\tau_2}}\Big|u_{(\xi_{\tau_2})}(x_{\tau_2})+(\xi^0_{\tau_2}+\xi_{\tau_2}^{d+1})u(x_{\tau_1^\delta})\Big|\\
&+|f|_{1,D}E\int_0^{\tau_2}e^{-\phi_s}\Big(|\xi_s|+2r_s+|\xi_s^0|+|\xi_s^{d+1}|\Big)ds\\
\le&Ee^{-\frac{1}{2}\phi_{\tau_2}}\Big|u_{(\xi_{\tau_2})}(x_{\tau_2})\Big|+\Big(|g|_{0,D}+|f|_{0,D}\Big)\Big(Ee^{-\frac{1}{2}\phi_{\tau_2}}|\xi_{\tau_2}^0|+Ee^{-\phi_{\tau_2}}|\xi_{\tau_2}^{d+1}|\Big)\\
&+|f|_{1,D}\bigg(E\int_0^{\tau_2}e^{-\phi_s}\big(|\xi_s|+2r_s\big)ds+4E\sup_{t\le\tau_2}e^{-\frac{3}{4}\phi_t}|\xi_t^{d+1}|+2E\sup_{s\le{\tau_2}}e^{-\frac{1}{2}\phi_{s}}|\xi_{s}^0|\bigg).
\end{align*}
By Lemma \ref{3l2}, Davis inequality and H\"older inequality,
\allowdisplaybreaks\begin{align*}
Ee^{-\frac{1}{2}\phi_{\tau_2}}|u_{(\xi_{\tau_2})}(x_{\tau_2})|\le&\sup_{ x\in \partial D_{\lambda^2}}\uxosbt\cdot E\sqrt{e^{-\phi_{\tau_2}}\mathrm{B}_2(x_{\tau_2},\xi_{\tau_2})}\\
\le&\sup_{x\in \partial D_{\lambda^2}}\uxosbt\cdot \sbtxzxz,\\
Ee^{-\frac{1}{2}\phi_{\tau_2}}|\xi_{\tau_2}^0|\le E\sup_{s\le{\tau_2}}e^{-\frac{1}{2}\phi_{s}}|\xi_{s}^0|=&E\sup_{s\le\tau_2}\Big|\int_0^{s}e^{-\frac{1}{2}\phi_s}\pi_rdw_r\Big|\\
\le&3E\Big(\int_0^{\tau_2}e^{-\phi_{r}}|\pi_r|^2dr\Big)^{\frac{1}{2}}\\
\le&NE\Big(\int_0^{\tau_2}e^{-\phi_r}|\xi_r|^2dr\Big)^\frac{1}{2}\\
\le&N\sbtxzxz,\\
Ee^{-\phi_{\tau_2}}|\xi_{\tau_2}^{d+1}|\le E\sup_{t\le\tau_2}e^{-\frac{3}{4}\phi_t}|\xi_t^{d+1}|\le&NE\int_0^{\tau_2}e^{-\frac{3}{4}\phi_s}|\xi_s|ds\\
\le&NE\Big(\int_0^{\tau_2}e^{-\frac{1}{2}\phi_s}ds\Big)^\frac{1}{2}\Big(\int_0^{\tau_2}e^{-\phi_s}|\xi_s|^2ds\Big)^\frac{1}{2}\\
\le&N\bigg(E\int_0^{\tau_2}e^{-\phi_s}|\xi_s|^2ds\bigg)^\frac{1}{2}\\
\le&N\sbtxzxz,\\
E\int_0^{\tau_2}e^{-\phi_s}\big(|\xi_s|+2r_s\big)ds\le&NE\int_0^{\tau_2}e^{-\phi_s}|\xi_s|ds\le N\sbtxzxz.
\end{align*}
Collecting all estimates above, we conclude that
$$
\uxzxz\le\sup_{x\in \partial D_{\lambda^2}}\uxosbt\cdot \sbtxzxz+N(|g|_{0,D}+|f|_{1,D})\sbtxzxz.
$$
So for any $x_0\in D_{\lambda^2}$, $\xi_0\in\Rd\setminus\{0\}$, we have
\begin{equation}\label{3i}
\frac{\uxzxz}{\sbtxzxz}\le\sup_{x\in \partial D_{\lambda^2}}\uxosbt+N_1,
\end{equation}
with $N_1$ defined by (\ref{N1}).

Notice that
\begin{equation*}
B_1(x,\xi)\left\{
\begin{array}{ll}
\ge\sqrt\psi(1+\sqrt\psi)|\xi|^2\ge\lambda^\frac{1}{2}|\xi|^2&\mbox{on }\{\psi=\lambda\}\\
\displaystyle\le\lambda(2+\lambda)|\xi|^2+K_1(2\lambda^2)^\frac{3}{2}\frac{\psi_{(\xi)}^2}{\lambda^2}\le K\lambda|\xi|^2&\mbox{on }\{\psi=\lambda^2\}.
\end{array}
\right.  
\end{equation*}

Recall that $K$ doesn't depend on $\lambda$. So for sufficiently small $\lambda$, we have
$$B_1(x,\xi)\ge 4B_2(x,\xi)\mbox{ when }\psi=\lambda,\qquad 4B_1(x,\xi)\le B_2(x,\xi)\mbox{ when }\psi=\lambda^2.$$

Then on $\{x\in D:\psi(x)=\lambda\}$, we have
\begin{align*}
\uxosbw\le&\frac{1}{2}\uxosbt\\
\le&\frac{1}{2}(\sup_{\{\psi=\lambda^2\}}\uxosbt+N_1)\\
\le&\frac{1}{4}\sup_{\{\psi=\lambda^2\}}\uxosbw+\frac{N_1}{2}\\
\le&\frac{1}{4}(\sup_{\{\psi=\lambda\}}\uxosbw+\sup_{\{\psi=\delta\}}\uxosbw+N_1)+\frac{N_1}{2}\\
=&\frac{1}{4}\sup_{\{\psi=\lambda\}}\uxosbw+\frac{1}{4}\sup_{\{\psi=\delta\}}\uxosbw+\frac{3N_1}{4},
\end{align*}
which implies that
\begin{equation}\label{3j}
\sup_{\{\psi=\lambda\}}\uxosbw\le\frac{1}{3}\sup_{\{\psi=\delta\}}\uxosbw+N_1.
\end{equation}
Meanwhile, on $\{x\in D:\psi(x)=\lambda^2\}$, we have

\begin{align*}
\uxosbt\le&\frac{1}{2}\uxosbw\\
\le&\frac{1}{2}(\sup_{\{\psi=\lambda\}}\uxosbw+\sup_{\{\psi=\delta\}}\uxosbw+N_1)\\
\le&\frac{1}{2}(\frac{1}{3}\sup_{\{\psi=\delta\}}\uxosbw+N_1)+\frac{1}{2}\sup_{\{\psi=\delta\}}\uxosbw+\frac{N_1}{2}\\
=&\frac{2}{3}\sup_{\{\psi=\delta\}}\uxosbw+N_1.
\end{align*}
Therefore, 
\begin{equation}\label{3k}
\sup_{\{\psi=\lambda^2\}}\uxosbt\le\frac{2}{3}\sup_{\{\psi=\delta\}}\uxosbw+N_1.
\end{equation}
Combining (\ref{3h}) and (\ref{3j}), we get, for any $x\in D_\delta^\lambda$, $\xi\in\Rd\setminus\{0\}$,
\begin{equation}\label{3l}
\uxosbw\le\frac{4}{3}\sup_{\{\psi=\delta\}}\uxosbw+2N_1.
\end{equation}
Combining (\ref{3i}) and (\ref{3k}), we get, for any $x\in D_{\lambda^2}$, $\xi\in\Rd\setminus\{0\}$,
\begin{equation}\label{3m}
\uxosbt\le\frac{2}{3}\sup_{\{\psi=\delta\}}\uxosbw+2N_1.
\end{equation}

Thus it remains to estimate
$$
\varlimsup_{\delta\downarrow0}\Big(\sup_{\{\psi=\delta\}}\uxosbw\Big).
$$
Notice that for each $\delta$, there exist $x(\delta)\in\{\psi=\delta\}$ and $\xi(\delta)\in\{\xi:|\xi|=1\}$, such that
$$
\sup_{\{\psi=\delta\}}\uxosbw=\frac{|u_{(\xi(\delta))}(x(\delta))|}{\sqrt{\mathrm{B}_1(x(\delta),\xi(\delta))}}.
$$
A subsequence of $(x(\delta),\xi(\delta))$ converges to some $(y,\zeta)$, such that $y\in\partial D$ and $|\zeta|=1$.

If $\psi_{(\zeta)}(y)\ne0$, then $\mathrm{B}_1(x(\delta),\xi(\delta))\rightarrow\infty$ as $\delta\downarrow0$. In this case,
$$
\varlimsup_{\delta\downarrow0}\Big(\sup_{\{\psi=\delta\}}\uxosbw\Big)=\varlimsup_{\delta\downarrow0}\frac{|u_{(\xi(\delta))}(x(\delta))|}{\sqrt{\mathrm{B}_1(x(\delta),\xi(\delta))}}=0.
$$

If $\psi_{(\zeta)}(y)=0$, then $\zeta$ is tangential to $\partial D$ at $y$. In this case,
\begin{equation}\label{lims}
\begin{gathered}
\varlimsup_{\delta\downarrow0}\Big(\sup_{\{\psi=\delta\}}\uxosbw\Big)=\varlimsup_{\delta\downarrow0}\frac{|u_{(\xi(\delta))}(x(\delta))|}{\sqrt{\mathrm{B}_1(x(\delta),\xi(\delta))}}\\
\qquad\qquad\qquad\qquad\qquad\ \ =\frac{|g_{(\zeta)}(y)|}{\sqrt\lambda}\le N\sup_{\partial D}|g_x|.
\end{gathered}
\end{equation}

From (\ref{3l}), (\ref{3m}) and (\ref{lims}), we have
$$
\frac{|u_{(\xi)}(x)|}{\sqrt{\mathrm{B}_1(x,\xi)}}\le N(|f|_{1,D}+|g|_{1,D}), \mbox{ when }x\in D^\lambda;
$$
$$
\frac{|u_{(\xi)}(x)|}{\sqrt{\mathrm{B}_2(x,\xi)}}\le N(|f|_{1,D}+|g|_{1,D}), \mbox{ when }x\in D_{\lambda^2}.
$$
Notice that $D^\lambda\cup D_{\lambda^2}=D$, and 
$$
\sqrt{\mathrm{B}_1(x,\xi)}\le N(|\xi|+\frac{|\psi_{(\xi)}|}{\psi^{\frac{1}{2}}}), \mbox{ when }x\in D^\lambda;
$$
$$
\sqrt{\mathrm{B}_2(x,\xi)}\le N(|\xi|+\frac{|\psi_{(\xi)}|}{\psi^{\frac{1}{2}}}), \mbox{ when }x\in D_{\lambda^2}.
$$
We conclude that, for any $x\in D$ and $\xi\in\Rd$,
$$|u_{(\xi)}(x)|\le N(|\xi|+\frac{|\psi_{(\xi)}|}{\psi^{\frac{1}{2}}})(|f|_{1,D}+|g|_{1,D}).$$
The inequality (\ref{3d}) is proved.
\end{proof}

The proof of (\ref{3dd}) is similar.

\begin{proof}[Proof of (\ref{3dd})]
When $x_0\in D_\delta^\lambda$, by Theorem \ref{2cc}, we have
%\begin{align*}
%u_{(\xi_0)(\xi_0)}(x_0)=E\bigg\{e^{-\phi_{\tau_1^\delta}}\Big[&u_{(\xi_{\tau_1^\delta})(\xi_{\tau_1^\delta})}(x_{\tau_1^\delta})+u_{(\eta_{\tau_1^\delta})}(x_{\tau_1^\delta})+2\xi_{\tau_1^\delta}^{d+1}u_{(\xi_{\tau_1^\delta})}(x_{\tau_1^\delta})\\
%&+\big(2\xi^0_{\tau_1^\delta}\xi^{d+1}_{\tau_1^\delta}+(\xi^{d+1}_{\tau_1^\delta})^2+\eta_{\tau_1^\delta}^{d+1}\big)u(x_{\tau_1^\delta})\Big]\\
%+\int_0^{\tau_1^\delta}e^{-\phi_s}\Big[&f_{(\xi_s)(\xi_s)}(x_s)+f_{(\eta_s)}(x_s)+\big(4r_s+2\xi_s^{d+1}\big)f_{(\xi_s)}(x_s)+\big(2\hat{r}_s\\
%&+4(\xi_s^0+\xi_s^{d+1})r_s+2\xi_s^0\xi_s^{d+1}+(\xi_s^{d+1})^2+\eta_s^{d+1}\big)f(x_s)\Big]ds\bigg\}
%\end{align*}
$$u_{(\xi_0)(\xi_0)}(x_0)=u_{(\xi_0)(\xi_0)}(x_0)+u_{(\eta_0)}(x_0)=Y_0=EY_{\tau_1^\delta}.$$
From (\ref{Zeta}) and (\ref{u}),
\begin{align*}
|u_{(\xi_0)(\xi_0)}(x_0)|\le &E|u_{(\xi_{\tau_1^\delta})(\xi_{\tau_1^\delta})}(x_{\tau_1^\delta})|+\sup_{x\in\partial D_\delta^\lambda,|\zeta|=1}|u_{(\zeta)}(x)|\cdot Ee^{-{\tau_1^\delta}}\Big(|\eta_{\tau_1^\delta}|+2|\tilde\xi^0_{\tau_1^\delta}||\xi_{\tau_1^\delta}|\Big)\\
+&\Big(|g|_{0,D}+|f|_{0,D}\Big)Ee^{-{\tau_1^\delta}}|\tilde\eta^0_{\tau_1^\delta}|+|f|_{2,D}E\int_0^{\tau_1^\delta} e^{-s}\Big[|\xi_s|^2+|\eta_s|\\
+&(4r_s+2|\tilde\xi^0_s|)|\xi_s|+2\hat r_s+4|\tilde\xi_s^0|r_s+|\tilde\eta^0_s|\Big]ds.
\end{align*}
Recall that in this case,
$$\tilde\xi_t^0=\xi^0_t+\xi_t^{d+1},\qquad\tilde\eta^0_t=2\xi_t^0\xi_t^{d+1}+(\xi_t^{d+1})^2+\eta_t^{d+1}.$$
It follows that
\begin{align*}
|u_{(\xi_0)(\xi_0)}(x_0)|\le &E|u_{(\xi_{\tau_1^\delta})(\xi_{\tau_1^\delta})}(x_{\tau_1^\delta})|+N\Big(|g|_{0,D}+|f|_{0,D}+\sup_{x\in\partial D_\delta^\lambda,|\zeta|=1}|u_{(\zeta)}(x)|\Big)\\
&\cdot Ee^{-{\tau_1^\delta}}\Big(|\eta_{\tau_1^\delta}|+|\xi_{\tau_1^\delta}|^2+|\xi^0_{\tau_1^\delta}|^2+|\xi_{\tau_1^\delta}^{d+1}|^2+|\eta_{\tau_1^\delta}^{d+1}|\Big)\\
+&N|f|_{2,D}E\int_0^{\tau_1^\delta} e^{-s}\Big[|\xi_s|^2+|\eta_s|+|\xi^0_{s}|^2+|\xi_{s}^{d+1}|^2+|\eta_{s}^{d+1}|+r_s^2+\hat r_s\Big]ds\\
\le&E|u_{(\xi_{\tau_1^\delta})(\xi_{\tau_1^\delta})}(x_{\tau_1^\delta})|+N\Big(|g|_{0,D}+|f|_{2,D}+\sup_{x\in\partial D_\delta^\lambda,|\zeta|=1}|u_{(\zeta)}(x)|\Big)\\
&\cdot\Big(E\sup_{t\le\tau_1^\delta}|\eta_t|+E\sup_{t\le\tau_1^\delta}|\xi_t|^2+E\sup_{t\le\tau_1^\delta}|\xi^0_t|^2+ E\sup_{t\le\tau_1^\delta}e^{-\frac{1}{2}t}|\xi_t^{d+1}|^2\\
&\ \ \ \ +E\sup_{t\le\tau_1^\delta}e^{-\frac{1}{2}t}|\eta_t^{d+1}|+E\int_0^{\tau_1^\delta}r_s^2+\hat{r}_sds\Big).
\end{align*}
By Lemma \ref{3l1}, Davis inequality and H\"older inequality,
\allowdisplaybreaks\begin{align}
\nonumber E|u_{(\xi_{\tau_1^\delta})(\xi_{\tau_1^\delta})}(x_{\tau_1^\delta})|\le&\sup_{ x\in \partial D_\delta^\lambda}\uxxobw\cdot E\mathrm{B}_1(x_{\tau_1^\delta},\xi_{\tau_1^\delta})\\
\nonumber\le&\sup_{x\in \partial D_\delta^\lambda}\uxxobw\cdot \bwxzxz,\\
E\sup_{t\le\tau_1^\delta}|\eta_t|\le&N\bwxzxz,\\
E\sup_{t\le\tau_1^\delta}|\xi_t|^2\le&N\bwxzxz,\\
E\sup_{t\le\tau_1^\delta}|\xi^0_t|^2\le&4E\langle\xi^0\rangle _{\tau_1^\delta}\le NE\int_0^{\tau_1^\delta}\pxtsops dt\le N\bwxzxz,\\
E\sup_{t\le\tau_1^\delta}e^{-\frac{1}{2}t}|\xi_t^{d+1}|^2\le&NE\sup_{t\le\tau_1^\delta}e^{-\frac{1}{2}t}\bigg(\int_0^t\Big(|\xi_s|+\frac{|\psi_{(\xi_s)}|}{\psi}\Big)ds\bigg)^2\\
\nonumber\le&NE\sup_{t\le\tau_1^\delta}e^{-\frac{1}{2}t}t\int_0^t\Big(|\xi_s|^2+\frac{\psi_{(\xi_s)}^2}{\psi^2} \Big)ds\\
\nonumber\le&NE\int_0^{\tau_1^\delta}\Big(|\xi_t|^2+\pxtsops\Big) dt\\
\nonumber\le&N\bwxzxz,\\
E\sup_{t\le\tau_1^\delta}e^{-\frac{1}{2}t}|\eta_t^{d+1}|\le&NE\sup_{t\le\tau_1^\delta}e^{-\frac{1}{2}t}\int_0^t\Big(|\xi_s|^2+\frac{\psi_{(\xi_s)}^2}{\psi^2}+|\eta_s|\Big)ds\\
\nonumber\le&NE\sup_{t\le\tau_1^\delta}e^{-\frac{1}{2}t}\bigg[\int_0^t\Big(|\xi_s|^2+\frac{\psi_{(\xi_s)}^2}{\psi^2}\Big)ds+\sqrt t\Big(\int_0^t|\eta_s|^2ds\Big)^\frac{1}{2}\bigg]\\
\nonumber\le&N\bigg[E\int_0^{\tau_1^\delta}\Big(|\xi_t|^2+\pxtsops\Big) dt+E\Big(\int_0^{\tau_1^\delta}|\eta_t|^2dt\Big)^2\bigg]\\
\nonumber\le&N\bwxzxz,\\
E\int_0^{\tau_1^\delta}(r_s^2+\hat{r}_s)ds\le&NE\int_0^{\tau_1^\delta}\Big(|\xi_t|^2+\frac{\psi_{(\xi_t)}^2}{\psi^2}\Big)dt\le N\bwxzxz.
\end{align}
Collecting all estimates above, we conclude that
\begin{align*}
\uxzxzxz\le&\sup_{x\in \partial D_\delta^\lambda}\uxxobw\cdot \bwxzxz\\
&+N\Big(|g|_{0,D}+|f|_{2,D}+\sup_{x\in\partial D_\delta^\lambda,|\zeta|=1}|u_{(\zeta)}(x)|\Big)\bwxzxz.
\end{align*}
So for any $x_0\in D_\delta^\lambda$, $\xi_0\in\Rd\setminus\{0\}$, we have
\begin{equation}\label{3o}
\frac{\uxzxzxz}{\bwxzxz}\le\sup_{x\in \partial D_\delta^\lambda}\uxxobw+N_2,
\end{equation}
with
\begin{equation}\label{N2}
N_2=N\Big(|g|_{2,D}+|f|_{2,D}+\sup_{x\in\partial D_\delta^\lambda,|\zeta|=1}|u_{(\zeta)}(x)|\Big).
\end{equation}

When $x_0\in D_{\lambda^2}$, by Theorem \ref{2cc}, we have
%\begin{align*}
%u_{(\xi_0)(\xi_0)}(x_0)=E\bigg\{e^{-\phi_{\tau_2}}\Big[&u_{(\xi_{\tau_2})(\xi_{\tau_2})}(x_{\tau_2})+u_{(\eta_{\tau_2})}(x_{\tau_2})+2\xi_{\tau_2}^{d+1}u_{(\xi_{\tau_2})}(x_{\tau_2})\\
%&+\big(\eta_{\tau_2}^0+2\xi_{\tau_2}^0\xi_{\tau_2}^{d+1}+(\xi^{d+1}_{\tau_2})^2+\eta_{\tau_2}^{d+1}\big)u(x_{\tau_2})\Big]\\
%+\int_0^{\tau_2}e^{-\phi_s}\Big[&f_{(\xi_s)(\xi_s)}(x_s)+f_{(\eta_s)}(x_s)+\big(4r_s+2\xi_s^{d+1}\big)f_{(\xi_s)}(x_s)+\big(2\hat{r}_s\\
%&+4(\xi_s^0+\xi_s^{d+1})r_s+\eta_s^0+2\xi_s^0\xi_s^{d+1}+(\xi_s^{d+1})^2+\eta_s^{d+1}\big)f(x_s)\Big]ds\bigg\}
%\end{align*}
$$u_{(\xi_0)(\xi_0)}(x_0)=u_{(\xi_0)(\xi_0)}(x_0)+u_{(\eta_0)}(x_0)=Y_0=EY_{\tau_2}.$$

Again, from (\ref{Zeta}) and (\ref{u}),
\begin{align*}
|u_{(\xi_0)(\xi_0)}(x_0)|\le &Ee^{-\phi_{\tau_2}}|u_{(\xi_{\tau_2})(\xi_{\tau_2})}(x_{\tau_2})|+\sup_{x\in\partial D_{\lambda^2},|\zeta|=1}|u_{(\zeta)}(x)|\cdot Ee^{-\phi_{\tau_2}}\Big(|\eta_{\tau_2}|+2|\tilde\xi^0_{\tau_2}||\xi_{\tau_2}|\Big)\\
+&\Big(|g|_{0,D}+|f|_{0,D}\Big)Ee^{-\phi_{\tau_2}}|\tilde\eta^0_{\tau_2}|+|f|_{2,D}E\int_0^{\tau_2} e^{-\phi_s}\Big[|\xi_s|^2+|\eta_s|\\
+&(4r_s+2|\tilde\xi^0_s|)|\xi_s|+2\hat r_s+4|\tilde\xi_s^0|r_s+|\tilde\eta^0_s|\Big]ds.
\end{align*}
Recall that in this case,
$$\tilde\xi_t^0=\xi^0_t+\xi_t^{d+1},\qquad\tilde\eta^0_t=\eta_t^0+2\xi_t^0\xi_t^{d+1}+(\xi_t^{d+1})^2+\eta_t^{d+1}.$$
Also, notice that by (\ref{3d})
$$\sup_{x\in\partial D_{\lambda^2},|\zeta|=1}|u_{(\zeta)}(x)|\le N\Big(1+\frac{|\psi|_{1,D}}{\lambda^2}\Big)\Big(|f|_{1,D}+|g|_{1,D}\Big)\le N\Big(|f|_{1,D}+|g|_{1,D}\Big).$$
Therefore, 
\begin{align*}
|u_{(\xi_0)(\xi_0)}(x_0)|\le &Ee^{-\phi_{\tau_2}}|u_{(\xi_{\tau_2})(\xi_{\tau_2})}(x_{\tau_2})|+N\Big(|g|_{1,D}+|f|_{1,D}\Big)\\
&\cdot Ee^{-\phi_{\tau_2}}\Big(|\eta_{\tau_2}|+|\xi_{\tau_2}|^2+|\xi^0_{\tau_2}|^2+|\xi_{\tau_2}^{d+1}|^2+|\eta_{\tau_2}^{d+1}|+|\eta^0_{\tau_2}|\Big)\\
+&N|f|_{2,D}E\int_0^{\tau_2} e^{-\phi_s}\Big[|\xi_s|^2+|\eta_s|+|\xi^0_{s}|^2+|\xi_{s}^{d+1}|^2+|\eta_{s}^{d+1}|+|\eta^0_s|+r_s^2+\hat r_s\Big]ds\\
\le&Ee^{-\phi_{\tau_2}}|u_{(\xi_{\tau_2})(\xi_{\tau_2})}(x_{\tau_2})|+N\Big(|g|_{1,D}+|f|_{2,D}\Big)\\
&\cdot\Big(E\sup_{t\le\tau_2}e^{-\phi_t+\frac{1}{2}t}|\eta_t|+E\sup_{t\le\tau_2}e^{-\phi_t+\frac{1}{2}t}|\xi_t|^2+E\sup_{t\le\tau_2}e^{-\phi_t+\frac{1}{2}t}|\xi^0_t|^2\\
&\ \ \ \ + E\sup_{t\le\tau_2}e^{-\phi_t+\frac{1}{4}t}|\xi_t^{d+1}|^2+E\sup_{t\le\tau_2}e^{-\phi_t+\frac{1}{4}t}|\eta_t^{d+1}|\\
&\ \ \ \ +E\sup_{t\le\tau_2}e^{-\phi_t+\frac{1}{2}t}|\eta^0_t|+E\int_0^{\tau_2}e^{-\phi_s}\big(r_s^2+\hat{r}_s\big)ds\Big).
\end{align*}
By Lemma \ref{3l2}, Davis inequality and H\"older inequality,
\allowdisplaybreaks\begin{align*}
Ee^{-\phi_{\tau_2}}|u_{(\xi_{\tau_2})(\xi_{\tau_2})}(x_{\tau_2})|\le&\sup_{ x\in \partial D_{\lambda^2}}\uxxobt\cdot Ee^{-\phi_{\tau_2}}\mathrm{B}_2(x_{\tau_2},\xi_{\tau_2})\\
\le&\sup_{x\in \partial D_{\lambda^2}}\uxxobt\cdot \btxzxz,\\
E\sup_{t\le\tau_2}e^{-\phi_t+\frac{1}{2}t}|\eta_t|\le&N\btxzxz,\\
E\sup_{t\le\tau_2}e^{-\phi_t+\frac{1}{2}t}|\xi_t|^2\le&N\btxzxz,\\
E\sup_{t\le\tau_2}e^{-\phi_t+\frac{1}{2}t}|\xi^0_t|^2=&E\sup_{t\le\tau_2}\bigg|\int_0^te^{-\frac{1}{2}\phi_t+\frac{1}{4}t}\pi_sdw_s\bigg|^2\\
\le&4E\int_0^{\tau_2}e^{-\phi_t+\frac{1}{2}t}|\pi_t|^2dt\\
\le&NE\int_0^{\tau_2}e^{-\phi_t+\frac{1}{2}t}|\xi_t|^2dt\\
\le&N\btxzxz,\\
E\sup_{t\le\tau_2}e^{-\phi_t+\frac{1}{4}t}|\xi_t^{d+1}|^2\le&NE\sup_{t\le\tau_2}e^{-\phi_t+\frac{1}{4}t}\bigg(\int_0^t|\xi_s|ds\bigg)^2\\
\le&NE\sup_{t\le\tau_2}e^{-\frac{1}{4}t}\bigg(\int_0^te^{-\frac{1}{2}\phi_s+\frac{1}{4}s}|\xi_s|ds\bigg)^2\\
\le&NE\sup_{t\le\tau_2}e^{-\frac{1}{4}t}\cdot t\int_0^te^{-\phi_s+\frac{1}{2}s}|\xi_s|^2ds\\
\le&NE\int_0^{\tau_2}e^{-\phi_s+\frac{1}{2}s}|\xi_s|^2ds\\
\le&N\btxzxz,\\
E\sup_{t\le\tau_2}e^{-\phi_t+\frac{1}{4}t}|\eta_t^{d+1}|\le&NE\sup_{t\le\tau_2}e^{-\phi_t+\frac{1}{4}t}\int_0^t|\xi_s|^2+|\eta_s|ds\\
\le&NE\int_0^{\tau_2}e^{-\phi_s+\frac{1}{4}s}|\xi_s|^2ds+NE\sup_{t\le\tau_2}e^{-\frac{1}{4}t}\int_0^te^{-\phi_s+\frac{1}{2}s}|\eta_s|ds\\
\le&N\btxzxz+NE\sup_{t\le\tau_2}e^{-\frac{1}{4}t}\cdot\sqrt t\bigg(\int_0^te^{-2\phi_s+s}|\eta_s|^2ds\bigg)^\frac{1}{2}\\
\le&N\btxzxz+NE\bigg(\int_0^{\tau_2}e^{-2\phi_s+s}|\eta_s|^2ds\bigg)^\frac{1}{2}\\
\le&N\btxzxz,\\
E\sup_{t\le\tau_2}e^{-\phi_t+\frac{1}{2}t}|\eta^0_t|\le&E\sup_{t\le\tau_2}\bigg(\Big|\int_0^te^{-\frac{1}{2}\phi_t+\frac{1}{4}t}\pi_sdw_s\Big|^2+\int_0^te^{-\phi_t+\frac{1}{2}t}|\pi_s|^2ds\\
&+\Big|\int_0^te^{-\phi_t+\frac{1}{2}t}\hat{\pi}_sdw_s\Big|\bigg)\\
\le&5E\int_0^{\tau_2}e^{-\phi_s+\frac{1}{2}s}|\pi_s|^2ds+3E\Big(\int_0^{\tau_2}e^{-2\phi_t+t}|\hat\pi_s|^2\Big)^\frac{1}{2}ds\\
\le&N\btxzxz,\\
E\int_0^{\tau_2}e^{-\phi_s}\big(r_s^2+\hat{r}_s\big)ds\le &NE\int_0^{\tau_2}e^{-\phi_s}\big(|\xi_s|^2+|\eta_s|\big)ds\\
\le& N\btxzxz+N\bigg(\int_0^{\tau_2}e^{-s}ds\bigg)^\frac{1}{2}\bigg(\int_0^{\tau_2}e^{-2\phi_s+s}|\eta_s|^2ds\bigg)^\frac{1}{2}\\
\le &N\btxzxz.
\end{align*}
Collecting all estimates above, we conclude that
$$
\uxzxzxz\le\sup_{x\in \partial D_{\lambda^2}}\uxxobt\cdot \btxzxz+N(|g|_{1,D}+|f|_{2,D})\btxzxz.
$$
So for any $x_0\in D_{\lambda^2}$, $\xi_0\in\Rd\setminus\{0\}$, we have
\begin{equation}\label{3p}
\frac{\uxzxzxz}{\btxzxz}\le\sup_{x\in \partial D_{\lambda^2}}\uxxobt+N_2,
\end{equation}
with $N_2$ defined by (\ref{N2}).

Then on $\{x\in D:\psi(x)=\lambda\}$, we have
\begin{align*}
\uxxobw\le&\frac{1}{4}\uxxobt\\
\le&\frac{1}{4}(\sup_{\{\psi=\lambda^2\}}\uxxobt+N_2)\\
\le&\frac{1}{16}\sup_{\{\psi=\lambda^2\}}\uxxobw+\frac{N_2}{4}\\
\le&\frac{1}{16}(\sup_{\{\psi=\lambda\}}\uxxobw+\sup_{\{\psi=\delta\}}\uxxobw+N_2)+\frac{N_2}{4}\\
=&\frac{1}{16}\sup_{\{\psi=\lambda\}}\uxxobw+\frac{1}{16}\sup_{\{\psi=\delta\}}\uxxobw+\frac{5N_2}{16},
\end{align*}
which implies that
\begin{equation}\label{3q}
\sup_{\{\psi=\lambda\}}\uxxobw\le\frac{1}{15}\sup_{\{\psi=\delta\}}\uxxobw+\frac{N_2}{3}.
\end{equation}
Meanwhile, on $\{x\in D:\psi(x)=\lambda^2\}$, we have

\begin{align*}
\uxxobt\le&\frac{1}{4}\uxxobw\\
\le&\frac{1}{4}(\sup_{\{\psi=\lambda\}}\uxxobw+\sup_{\{\psi=\delta\}}\uxxobw+N_2)\\
\le&\frac{1}{4}(\frac{1}{15}\sup_{\{\psi=\delta\}}\uxxobw+\frac{N_2}{3})+\frac{1}{4}\sup_{\{\psi=\delta\}}\uxxobw+\frac{N_2}{4}\\
=&\frac{4}{15}\sup_{\{\psi=\delta\}}\uxxobw+\frac{N_2}{3}.
\end{align*}
Therefore,
\begin{equation}\label{3r}
\sup_{\{\psi=\lambda^2\}}\uxxobt\le\frac{4}{15}\sup_{\{\psi=\delta\}}\uxxobw+\frac{N_2}{3}.
\end{equation}
Combining (\ref{3o}) and (\ref{3q}), we get, for any $x\in D_\delta^\lambda$, $\xi\in\Rd\setminus\{0\}$,
\begin{equation}\label{3s}
\uxxobw\le\frac{16}{15}\sup_{\{\psi=\delta\}}\uxxobw+\frac{4N_2}{3}.
\end{equation}
Combining (\ref{3p}) and (\ref{3r}), we get, for any $x\in D_{\lambda^2}$, $\xi\in\Rd\setminus\{0\}$,
\begin{equation}\label{3t}
\uxxobt\le\frac{4}{15}\sup_{\{\psi=\delta\}}\uxxobw+\frac{4N_2}{3}.
\end{equation}

Thus it remains to estimate
$$
\varlimsup_{\delta\downarrow0}\Big(\sup_{\{\psi=\delta\}}\uxxobw\Big) \mbox{ and }\varlimsup_{\delta\downarrow0}\sup_{x\in\partial D_\delta^\lambda,|\zeta|=1}|u_{(\zeta)}(x)|.
$$

First, notice that
\begin{align*}
\varlimsup_{\delta\downarrow0}\sup_{x\in\partial D_\delta^\lambda,|\zeta|=1}|u_{(\zeta)}(x)|\le&\sup_{x\in\partial D,|\zeta|=1}|u_{(\zeta)}(x)|+\sup_{x\in\{\psi=\lambda\},|\zeta|=1}|u_{(\zeta)}(x)|\\
\le&\sup_{x\in\partial D,|l|=1,l\parallel \partial D}|u_{(l)}(x)|+\sup_{x\in\partial D,|n|=1,n\perp \partial D}|u_{(n)}(x)|\\
&+\sup_{x\in\{\psi=\lambda\},|\zeta|=1}|u_{(\zeta)}(x)|,
\end{align*}
Apply Lemma \ref{lemma4} and first derivative estimate (\ref{3d}), we get
\begin{align*}
\varlimsup_{\delta\downarrow0}\sup_{x\in\partial D_\delta^\lambda,|\zeta|=1}|u_{(\zeta)}(x)|\le&\sup_{x\in\partial D,|l|=1,l\parallel \partial D}|g_{(l)}(x)|+N(|g|_{2,D}+|f|_{0,D})\\
&+N\Big(1+\frac{|\psi|_{1,D}}{\lambda}\Big)(|g|_{1,D}+|f|_{1,D})\\
\le&N(|g|_{2,D}+|f|_{1,D}).
\end{align*}

Second, notice that for each $\delta$, there exist $x(\delta)\in\{\psi=\delta\}$ and $\xi(\delta)\in\{\xi:|\xi|=1\}$, such that
$$
\sup_{\{\psi=\delta\}}\uxxobw=\frac{|u_{(\xi(\delta))(\xi(\delta))}(x(\delta))|}{\mathrm{B}_1(x(\delta),\xi(\delta))}.
$$
A subsequence of $(x(\delta),\xi(\delta))$ converges to some $(y,\zeta)$, such that $y\in\partial D$ and $|\zeta|=1$.

If $\psi_{(\zeta)}(y)\ne0$, then $\mathrm{B}_1(x(\delta),\xi(\delta))\rightarrow\infty$ as $\delta\downarrow0$. In this case,
$$
\varlimsup_{\delta\downarrow0}\Big(\sup_{\{\psi=\delta\}}\uxxobw\Big)=\varlimsup_{\delta\downarrow0}\frac{|u_{(\xi(\delta))(\xi(\delta))}(x(\delta))|}{\mathrm{B}_1(x(\delta),\xi(\delta))}=0.
$$

If $\psi_{(\zeta)}(y)=0$, then $\zeta$ is tangential to $\partial D$ at $y$. In this case,
$$
\varlimsup_{\delta\downarrow0}\Big(\sup_{\{\psi=\delta\}}\uxxobw\Big)=\varlimsup_{\delta\downarrow0}\frac{|u_{(\xi(\delta))(\xi(\delta))}(x(\delta))|}{\mathrm{B}_1(x(\delta),\xi(\delta))}=\frac{|g_{(\zeta)(\zeta)}(y)|+K|u_{(n)}(y)|}{\lambda}.$$
By Lemma (\ref{lemma4}), we have
\begin{equation*}
\frac{|g_{(\zeta)(\zeta)}(y)|+K|u_{(n)}(y)|}{\lambda}\le N(|g|_{2,D}+|f|_{0,D}).
\end{equation*}
Therefore, we have
$$
\frac{|u_{(\xi)(\xi)}(x)|}{\mathrm{B}_1(x,\xi)}\le N(|f|_{2,D}+|g|_{2,D}), \mbox{ when }x\in D^\lambda;
$$
$$
\frac{|u_{(\xi)(\xi)}(x)|}{\mathrm{B}_2(x,\xi)}\le N(|f|_{2,D}+|g|_{2,D}), \mbox{ when }x\in D_{\lambda^2}.
$$
It follows that, for any $x\in D$ and $\xi\in\Rd$,
$$|u_{(\xi)(\xi)}(x)|\le N(|\xi|^2+\frac{\psi_{(\xi)}^2}{\psi})(|f|_{2,D}+|g|_{2,D}).$$
The inequality (\ref{3dd}) is proved.

\end{proof}

\begin{proof}[Proof of the existence and uniqueness of (\ref{solva})]
The fact that $u$ given by (\ref{1b}) satisfies (\ref{solva}) follows from Theorem 1.3 in \cite{MR617995}.

To prove the uniqueness, assume that $u_1, u_2\in C_{loc}^{1,1}(D)\cap C^{0,1}(\bar D)$ are solutions of (\ref{solva}).  Let $\Lambda=|u_1|_{0,D}\vee|u_2|_{0,D}$. For constants $\delta$ and $\ve$ satisfying $0<\delta<\ve<1$, define
$$\Psi(x,t)=\ve(1+\psi(x))\Lambda e^{-\delta t},\ U(x,t)=u(x) e^{-\ve t}\mbox{ in }\bar D\times(0,\infty),$$
$$F[U]=U_t+LU-cU+f \mbox{ in } D\times (0,\infty).$$

Notice that a.e. in $D$, we have
$$F[U_1-\Psi]=-\ve e^{-\ve t}u_1+\delta\Psi-\ve\Lambda e^{-\delta t}L\psi+c\Psi\ge\ve\Lambda(e^{-\delta t}-e^{-\ve t})\ge0, $$
$$F[U_2+\Psi]=\ve e^{-\ve t}u_2-\delta\Psi+\ve\Lambda e^{-\delta t}L\psi-c\Psi\le\ve\Lambda(e^{-\ve t}-e^{-\delta t})\le0.$$
On $\partial D\times (0,\infty)$, we have
$$U_1-U_2-2\Psi=-2\Psi\le0.$$
On $\bar D\times T$, where $T=T(\ve,\delta)$ is a sufficiently large constant, we have
$$U_1-U_2-2\Psi=(u_1-u_2)e^{-\ve T}-2\ve(1+\psi)\Lambda e^{-\delta T}\le2\Lambda(e^{-\ve T}-\ve e^{-\delta T})\le 0.$$

Applying Theorem 1.1 in \cite{MR538554}, we get
$$U_1-U_2-2\Psi\le 0 \mbox { a.e. in } \bar D\times(0,T).$$
It follows that
$$u_1-u_2\le 2\ve(1+\psi)\Lambda e\rightarrow0, \mbox{ as } \ve\rightarrow0, \mbox{ a.e. in }D.$$
Similarly, $u_2-u_1\le0$ a.e. in $D$. The uniqueness is proved.

\end{proof}

\begin{remark}
Based on our proof, if we replace $\sigma(x), b(x), c(x), f(x)$ and $g(x)$ in (\ref{1aa}) and (\ref{1b}) by $\sigma(\omega,t,x),b(\omega,t,x),c(\omega,t,x), f(\omega,t,x)$ and $g(\omega,t,x)$, defined on $\Omega\times[0,\infty)\times D$, under appropriate measurable assumptions, the first and second derivative estimates (\ref{3d}) and (\ref{3dd}) are still true.
\end{remark}

\section*{Acknowledgements} The author is sincerely grateful to his advisor, N.V. Krylov, for introducing this method to the author and giving many useful suggestions and comments on the improvements. The author also would like to thank the anonymous referees whose helpful comments improved the quality of this manuscript.

%\bibliographystyle{amsplain}
%\bibliography{reference}

\end{document}